\documentclass{article}

\usepackage{geometry}
\usepackage{makeidx}         
\usepackage{graphicx}        
\usepackage{multicol}        
\usepackage[bottom]{footmisc}

\usepackage{color}

\usepackage{epsfig}
\usepackage{wrapfig}
\usepackage{subfig}

\usepackage{enumitem}

\usepackage{amsmath,amsfonts,amssymb}
\usepackage{hyperref}

\newcommand{\abs}[1]{|#1|}
\newcommand{\cK}{\mathcal{K}}
\newcommand{\dd}[1][x]{\,\operatorname{d}\!#1}
\newcommand{\DD}{\operatorname{D}\!}
\newcommand{\integral}[3]{\int_{#1} \; {#2} \; \dd[#3]\,}
\newcommand{\integrall}[4]{\int_{#1}^{#2} \; {#3} \; \dd[#4]\,}
\newcommand{\integralm}[3]{\int_{#1} \; {#2} \; {#3}\,}

\newcommand{\N}{\mathbb{N}}
\newcommand{\R}{\mathbb{R}}

\newcommand{\set}[2]{\{\, {#1} \,|\, {#2} \,\}}

\newcommand{\sdiff}[2]{\partial_{#2} {#1}} 
\newcommand{\sdifff}[3]{\partial^{#3}_{#2} {#1}}

\newcommand{\Diff}[2]{\frac{\dd[#1]}{\dd[#2]}}

\newcommand{\Caputo}[1]{\mathcal{D}^\alpha_{#1}}
\newcommand{\Caputoo}[2]{\mathcal{D}^{#2}_{#1}}
\newcommand{\Dalpha}[1][\alpha]{\Caputoo{+}{#1}}

\newcommand{\Fourier}{\operator{F}} 
\newcommand{\FVar}{k} 
\newcommand{\RFalpha}{a}
\newcommand{\RFtheta}{\theta}
\newcommand{\RFsymbol}{\psi^{\RFalpha}_{\RFtheta}}
\newcommand{\RieszFeller}{D^{\RFalpha}_{\RFtheta}}
\newcommand{\RieszFellerII}[2]{D^{#1}_{#2}}
\newcommand{\Riesz}{\RieszFellerII{\RFalpha}{0}}
\newcommand{\Green}{G^{\RFalpha}_{\RFtheta}}  
\newcommand{\GreenII}[2]{G^{#1}_{#2}}  
\newcommand{\RFall}{\mathfrak{D}_{\RFalpha,\RFtheta}}  
\newcommand{\RFnonlocal}{\mathfrak{D}_{\RFalpha,\RFtheta}^\bullet}  
\newcommand{\RFnontrivial}{\mathfrak{D}_{\RFalpha,\RFtheta}^{\diamond}}  
\newcommand{\RFnonextremal}{\mathfrak{D}_{\RFalpha,\RFtheta}^{+}}  
\newcommand{\Levy}{\operator{L}}
\newcommand{\operator}[1]{\mathcal{#1}}
\DeclareMathOperator{\sgn}{sgn}
\newcommand{\euler}{\operatorname{e}}
\newcommand{\ii}{\operatorname{i}}

\DeclareMathOperator{\dom}{dom}

\newcommand{\um}{u_-}
\newcommand{\up}{u_+}
\newcommand{\upm}{u_\pm}

\newcommand{\profile}{\bar{u}}

\newcommand{\uStar}{u_*}
\newcommand{\ShockTriple}{(\um,\up;\speed)}
\newcommand{\TWS}{(\profile,\speed)}
\newcommand{\RDspeed}{c}
\newcommand{\speed}{c}

\newcommand{\xx}[1]{\,\text{ #1 }\,}
\newcommand{\Xx}[1]{\quad\text{ #1 }\,}
\newcommand{\XX}[1]{\quad\text{ #1 }\quad}

\newcommand{\SchwartzTF}{\mathcal{S}}

\newcommand{\FourierInv}{\mathcal{F}^{-1}}
                       
\usepackage{amsthm}
\theoremstyle{plain} 
\newtheorem{theorem}{Theorem}
\newtheorem{lemma}[theorem]{Lemma}
\newtheorem{proposition}{Proposition}

\theoremstyle{definition}
\newtheorem{definition}{Definition}

\theoremstyle{remark}
\newtheorem{remark}{Remark}
\newtheorem{example}{Example}

\begin{document}

\title{Two classes of nonlocal Evolution Equations related by a shared Traveling Wave Problem}
\author{Franz Achleitner \footnote{TU Wien, Institute for Analysis and Scientific Computing, Wiedner Hauptstrasse 8-10, A-1040 Wien, Austria, \href{mailto:franz.achleitner@tuwien.ac.at}{franz.achleitner@tuwien.ac.at}}}

\maketitle

\abstract{
We consider reaction-diffusion equations and Korteweg-de Vries-Burgers (KdVB) equations,
 i.e. scalar conservation laws with diffusive-dispersive regularization.
We review the existence of traveling wave solutions for these two classes of evolution equations.
For classical equations the traveling wave problem (TWP) for a local KdVB equation 
 can be identified with the TWP for a reaction-diffusion equation.
In this article we study this relationship for these two classes of evolution equations with nonlocal diffusion/dispersion.
This connection is especially useful, if the TW equation is not studied directly,
 but the existence of a TWS is proven using one of the evolution equations instead.
Finally, we present three models from fluid dynamics 
 and discuss the TWP via its link to associated reaction-diffusion equations.
}

\section{Introduction}
We will consider two classes of (nonlocal) evolution equations 
 and study the associated traveling wave problems in parallel:
On the one hand, we consider scalar conservation laws with (nonlocal) diffusive-dispersive regularization
 \begin{equation} \label{KdVB:f:nonlocal}
   \sdiff{u}{t} + \sdiff{}{x} f(u) = \epsilon \Levy_1[u] + \delta \sdiff{}{x} \Levy_2[u] \ , \quad t>0 \ ,\quad x\in\R \ ,
 \end{equation}
 for some nonlinear function $f:\R\to\R$, L\'evy operators $\Levy_1$ and $\Levy_2$,
 as well as constants $\epsilon,\delta\in\R$.
The Fourier multiplier operators $\Levy_1$ and $\sdiff{}{x} \Levy_2$ model diffusion and dispersion, respectively. 
On the other hand, we consider scalar reaction-diffusion equations
 \begin{equation} \label{RD:r:nonlocal}
   \sdiff{u}{t} = r(u) + \sigma \Levy_3[u] \ , \quad t>0 \ ,\quad x\in\R \ ,
 \end{equation}
 for some positive constant $\sigma$, as well as a nonlinear function $r:\R\to\R$ and a L\'evy operator $\Levy_3$
 modeling reaction and diffusion, respectively.

\begin{definition} \label{def:TWS}
A traveling wave solution (TWS) of an evolution equation--such as~\eqref{KdVB:f:nonlocal} and~\eqref{RD:r:nonlocal}--is
 a solution $u(x,t)=\profile(\xi)$ whose \emph{profile} $\profile$ depends on $\xi:=x-\speed t$ for some \emph{wave speed} $\speed$.
Moreover, the profile $\profile\in C^2(\R)$ is assumed to approach distinct endstates~$\upm$ such that
\begin{equation}\label{TWS:limits}
 \lim_{\xi\to\pm\infty} \profile(\xi) =\upm \ , \qquad
 \lim_{\xi\to\pm\infty} \profile^{(n)}(\xi) =0 \quad \text{with $n=1,2$.}
\end{equation}
\end{definition}
Such a TWS is also known as a \emph{front} in the literature.
A TWS $(\profile,\RDspeed)$ is called monotone, if its profile $\profile$ is a monotone function. 
\begin{definition} \label{def:TWP}
The traveling wave problem (TWP) associated to an evolution equation
 is to study for some distinct endstates $\upm$
 the existence of a TWS~$\TWS$ in the sense of Definition~\ref{def:TWS}.
\end{definition}
We want to identify classes of evolutions equations of type~\eqref{KdVB:f:nonlocal} and~\eqref{RD:r:nonlocal},
 which lead to the same TWP. 
This connection is especially useful, if the TWP is not studied directly,
 but the existence of a TWS is proven using one of the evolution equations instead.
A classical example of~\eqref{KdVB:f:nonlocal} is a scalar conservation law with local diffusive-dispersive regularization
 \begin{equation} \label{KdVB:f:local}
   \sdiff{u}{t} + \sdiff{}{x} f(u) = \epsilon \sdifff{u}{x}{2} +\delta \sdifff{u}{x}{3} \ , \quad t>0 \ ,\quad x\in\R \ ,
 \end{equation}
 for some nonlinear function $f:\R\to\R$ and some constants $\epsilon>0$ and $\delta\in\R$.
Equation~\eqref{KdVB:f:local} with Burgers flux $f(u)=u^2$ is known as Korteweg-de Vries-Burgers (KdVB) equation;
 hence we refer to Equation~\eqref{KdVB:f:local} with general $f$ as \emph{generalized} KdVB equation
 and Equation~\eqref{KdVB:f:nonlocal} as \emph{nonlocal generalized} KdVB equation.
A TWS~$\TWS$ satisfies the traveling wave equation (TWE)
 \begin{equation}\label{intro:KdVB:TWE:local}
  -c \profile' + f'(\profile)\ \profile' = \epsilon \profile'' + \delta \profile''' \ , \quad \xi\in\R \ , 
 \end{equation}
or integrating on $(-\infty,\xi]$ and using~\eqref{TWS:limits},
 \begin{equation}\label{intro:KdVB:TWE:local:integrated}
  h(\profile) := f(\profile) -c \profile -(f(\um) -c \ \um) = \epsilon \profile' + \delta \profile'' \ , \quad \xi\in\R \ . 
 \end{equation}
However, the TW ansatz $v(x,t)=\profile(x -\epsilon t)$ for the scalar reaction-diffusion equation
 \begin{equation} \label{RD:h:local}
  \sdiff{v}{t} = -h(v) +\delta \sdifff{v}{x}{2} \ , \quad t>0 \ , \quad x\in\R \ , 
 \end{equation}
 leads to the same TWE~\eqref{intro:KdVB:TWE:local:integrated}
 except for a different interpretation of the parameters.
The traveling wave speeds in the TWP of~\eqref{KdVB:f:local} and~\eqref{RD:h:local} are~$\speed$ and~$\epsilon$, respectively.
For fixed parameters $\speed$, $\epsilon$, and $\delta$,
 the existence of a traveling wave profile $\profile$ satisfying~\eqref{TWS:limits} and~\eqref{intro:KdVB:TWE:local:integrated}
 reduces to the existence of a heteroclinic orbit for this ODE.
This is an example, where the existence of TWS is studied directly via the TWE.

A first example, where the TWE is not studied directly,
 is the TWP for a nonlocal KdVB equation~\eqref{KdVB:f:nonlocal} with
 $\Levy_1[u] =\sdifff{u}{x}{2}$ and $\Levy_2[u] = \phi_\epsilon \ast u - u$
 for some even non-negative function $\phi\in L^1(\R)$ with compact support and unit mass,
 where $\phi_\epsilon(\cdot) := \phi(\cdot/\epsilon) / \epsilon$ with $\epsilon>0$.
It has been derived as a model for phase transitions with long range interactions close to the surface,
 which supports planar TWS associated to undercompressive shocks of~\eqref{SCL}, see~\cite{Rohde:2005}. 
In particular, the TWP for a cubic flux function $f(u)=u^3$ is related  
 to the TWP for a reaction-diffusion equation~\eqref{RD:r:nonlocal} with $\Levy_3[u] = \Levy_2[u]$.
The existence of TWS for this reaction-diffusion equation has been proven
 via a homotopy of~\eqref{RD:r:nonlocal} to a classical reaction-diffusion model~\eqref{RD:h:local}, see~\cite{Bates+etal:1997}.

\emph{Outline.} 
In Section~\ref{sec:Levy} we collect background material on L\'evy operators~$\Levy$,
 which will model diffusion in our nonlocal evolution equations.
Special emphasize is given to convolution operators and Riesz-Feller operators.
In Section~\ref{sec:local} we review the classical results on the TWP for reaction-diffusion equations~\eqref{RD:h:local}
 and generalized Korteweg-de Vries-Burgers equation~\eqref{KdVB:f:local}.
We study their relationship in detail, especially the classification of function $h(u)$,
 which will be used again in Section~\ref{sec:nonlocal}.
In Section~\ref{sec:nonlocal}, first we review the results on TWP for nonlocal reaction-diffusion equations~\eqref{RD:r:nonlocal}
 with operators $\Levy_3$ of convolution type and Riesz-Feller type, respectively.
Finally, we study the example of nonlocal generalized Korteweg-de Vries-Burgers equation modeling a shallow water flow~\cite{Kluwick+etal:2010}
 and Fowler's equation modeling dune formation~\cite{Fowler:2002},
\begin{equation} \label{dune}
  \sdiff{u}{t} + \sdiff{}{x} u^2= \sdifff{}{x}{2} u -\sdiff{}{x} \Dalpha[1/3] u \ , \quad t>0 \ , \quad x\in\R \ ,
\end{equation}
 where $\Dalpha$ is a Caputo derivative.
In the Appendix, we collect background material on Caputo derivatives~$\Caputo{+}$
 and the shock wave theory for scalar conservation laws,
 which will explain the importance of the TWP for KdVB equations.

\emph{Notations.}
We use the conventions in probability theory,
 and define the Fourier transform $\Fourier$ and its inverse $\FourierInv$ for $g\in L^1(\R)$ and $x,k\in\R$ as 
\[
  \Fourier [g](\FVar) := \integral{\R}{ \euler^{+\ii \FVar x} g(x)}{x} \ ; \qquad
  \FourierInv [g](x)  := \tfrac{1}{2\pi} \integral{\R}{ \euler^{-\ii \FVar x} g(\FVar)}{\FVar} \ .
\]
In the following, $\Fourier$ and $\FourierInv$ will denote also their respective extensions to $L^2(\R)$.

\section{L\'evy Operators} \label{sec:Levy}
A L\'evy process is a stochastic process
 with independent and stationary increments 
 which is continuous in probability~\cite{Applebaum:2009, Jacob:2001, Sato:1999}.
Therefore a L\'evy process is characterized by its transition probabilities~$p(t,x)$, 
which evolve according to an evolution equation 
\begin{equation}
 \sdiff{p}{t} = \Levy p
\end{equation}
for some operator $\Levy$, also called a \emph{L\'evy operator}.
First, we define L\'evy operators on the function spaces
 $C_0(\R) :=\{ f\in C(\R) \,|\, \lim_{|x|\to\infty} f(x)=0 \}$ and
 $C^2_0(\R) :=\{ f,f',f'' \in C_0(\R)\}$.

\begin{definition}
The family of L\'evy operators in one spatial dimension consists of operators~$\Levy$ defined 
 for $f\in C^2_0(\R)$ as
 \begin{equation} \label{op:Levy} 
  \Levy f(x) = \tfrac12 A f''(x) 
   + \gamma\ f'(x)  
   + \integralm{\R}{ \big( f(x+y) - f(x) - y\ f'(x) 1_{(-1,1)}(y) \big) }{\nu( \dd[y])}
 \end{equation}
 for some constants $A\geq 0$ and $\gamma\in\R$,
 and a measure~$\nu$ on $\R$ satisfying
 \begin{equation*} 
    \nu(\{0\}) = 0 \quad \text{and} \quad \integralm{\R}{ \min(1,|y|^2)}{\nu(\dd[y])} < \infty \,.
   \end{equation*}
\end{definition}

\begin{remark} 
 The function $f(x+y) - f(x) - y\ f'(x) 1_{(-1,1)}(y)$ 
  is integrable with respect to $\nu$, 
 because it is bounded outside of any neighborhood of $0$ and 
 \[
  f(x+y) - f(x) - y\ f'(x) 1_{(-1,1)}(y) = O(|y|^2) \quad \text{as} \quad |y|\to 0
 \]
 for fixed $x$. 
The indicator function $c(y) =1_{(-1,1)}(y)$ is only one possible choice
 to obtain an integrable integrand.
More generally, let $c(y)$ be a bounded measurable function from $\R$ to $\R$ satisfying 
 $c(y) = 1+o(|y|)$ as $|y|\to 0$, and $c(y) = O(1/|y|)$ as $|y|\to \infty$.
Then~\eqref{op:Levy} is rewritten as 
 \begin{multline} \label{op:Levy:c}
  \Levy f(x) = \tfrac12 A f''(x) 
   + \gamma_c\ f'(x) 
   + \integralm{\R}{ \big( f(x+y) - f(x) - y\ f'(x) c(y) \big) }{\nu( \dd[y])} \ ,                
 \end{multline}
 with $\gamma_c = \gamma + \integralm{\R}{ y\ (c(y)-1_{(-1,1)}(y)) }{\nu(\dd[y])}$. 
 \newline 

Alternative choices for $c$: 
\begin{enumerate}
 \item[(c 0)]
  If a L\'evy measure $\nu$ satisfies
	 $\integralm{|y|< 1}{ |y|}{\nu(\dd[y])} < \infty$  
   then $c\equiv 0$ is admissible. 
 \item[(c 1)] 
  If a L\'evy measure $\nu$ satisfies  
   $\integralm{|y|> 1}{ |y|}{\nu(\dd[y])} < \infty$
  then $c\equiv 1$ is admissible.
\end{enumerate}
 We note that $A$ and $\nu$ are invariant no matter what function~$c$ we choose. 
\end{remark}

\textbf{Examples.}
\begin{enumerate}[label=(\alph*)]
\item The L\'evy operators
 \begin{equation} \label{IG:compoundPoisson}
   \Levy f = \integralm{\R}{ \big( f(x+y) - f(x) \big) }{\nu( \dd[y])}
 \end{equation}
 are infinitesimal generators associated to a compound Poisson process with finite L\'evy measure $\nu$ satisfying (c 0). 
The special case of $\nu(\dd[y]) = \phi(-y) \dd[y]$ for some function $\phi\in L^1(\R)$ yields
 \begin{equation} \label{IG:compoundPoisson:II}
   \Levy f (x) 
	         = \integralm{\R}{ \big( f(x+y) - f(x) \big) }{\phi(-y) \dd[y]}
	         = \big( \phi \ast f - \integral{\R}{\phi}{y}\ f\big) (x) \ . 
 \end{equation}
\item \emph{Riesz-Feller operators.}
The Riesz-Feller operators of order $\RFalpha$ and asymmetry $\RFtheta$ are defined as Fourier multiplier operators
\begin{equation} \label{eq:RF:transform}
 \Fourier [\RieszFeller f](\FVar) = \RFsymbol (\FVar)\ \Fourier [f] (\FVar) \ ,\qquad  \FVar\in\R \ ,
\end{equation}
with symbol $\RFsymbol(\FVar) = -|\FVar|^{\RFalpha} \exp\left[\ii\ \sgn (\FVar)\ \RFtheta \pi/2 \right]$ such that $(\RFalpha,\RFtheta) \in \RFall$ and
\[ \RFall := \set{ (\RFalpha,\RFtheta)\in\R^2 }{ 0<\RFalpha \leq 2 \ , \quad |\RFtheta| \leq \min\{\RFalpha, 2-\RFalpha\} }\ . \]

\begin{figure}
\begin{center}
 \includegraphics{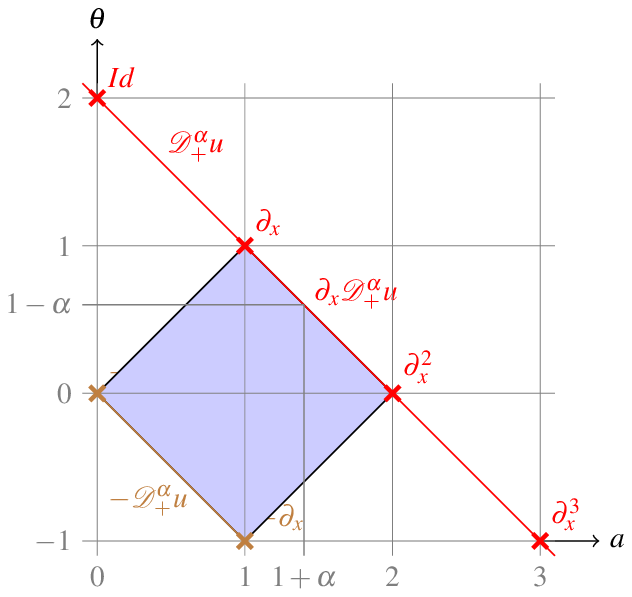}
\end{center}
\caption{
  The family of Fourier multipliers $\RFsymbol(\FVar) = -|\FVar|^{\RFalpha} \exp\big[\ii \sgn(\FVar) \RFtheta \pi /2)\big]$ 
  has two parameters $\RFalpha$ and $\RFtheta$.
  Some Fourier multiplier operators
	 $\Fourier[Tf](\FVar) = \RFsymbol(\FVar)\  \Fourier[f](\FVar)$
  are inserted in the parameter space $(\RFalpha,\RFtheta)$: 
	partial derivatives and Caputo derivatives~$\Caputo{+}$ with $0<\alpha<1$. 
  The Riesz-Feller operators~$\RieszFeller$ are those operators
	 with parameters $(\RFalpha,\RFtheta)\in\RFall$.
  The set~$\RFall$ is also called \emph{Feller-Takayasu diamond} and depicted as a shaded region,
   see also~\cite{Mainardi+Luchko+Pagnini:2001}.
}
\label{fig:RieszFeller}
\end{figure}

Special cases of Riesz-Feller operators are 
\begin{itemize}
\item Fractional Laplacians $-(-\Delta)^{\RFalpha/2}$ on $\R$ with Fourier symbol $-|k|^{\RFalpha}$ for $0<\RFalpha \leq 2$.
 In particular, fractional Laplacians are the only symmetric Riesz-Feller operators
  with $-(-\Delta)^{\RFalpha/2} = \RieszFellerII{\RFalpha}{0}$ and $\RFtheta\equiv 0$.
\item Caputo derivatives $-\Caputo{+}$ with $0<\alpha <1$ are Riesz-Feller operators with $\RFalpha =\alpha$ and $\theta =-\alpha$, such that $-\Caputo{+} =\RieszFellerII{\alpha}{-\alpha}$, see also Section~\ref{sec:Caputo}.
\item Derivatives of Caputo derivatives $\sdiff{}{x} \Caputo{+}$ with $0<\alpha <1$ are Riesz-Feller operators with $\RFalpha =1+\alpha$ and $\theta =1-\alpha$, such that $\sdiff{}{x} \Caputo{+} =\RieszFellerII{1+\alpha}{1-\alpha}$.
\end{itemize}
\end{enumerate}
Next we consider the Cauchy problem
\begin{equation} \label{eq:linearRF}
\sdiff{u}{t}(x,t) = \RieszFeller{} [u(\cdot ,t)](x) \ , \quad u(x,0) =u_0(x)\ ,
\end{equation}
for $(x,t)\in\R\times (0,\infty)$ and initial datum $u_0$.

\begin{proposition} \label{prop:MarkovSG}
For $(\RFalpha,\RFtheta)\in\RFall$ and $1\leq p<\infty$, 
 the Riesz-Feller operator~$\RieszFeller$ generates
 a strongly continuous $L^p$-semigroup
 \[
  S_t: L^p(\R) \to L^p(\R)\ , \quad u_0 \mapsto S_t u_0 = \Green(\cdot,t)\ast u_0 \ , 
 \]
 with heat kernel $\Green(x,t) =\FourierInv[\exp (t\ \RFsymbol(\cdot))](x)$.
In particular, $\Green(x,t)$ is the probability measure of a L\'{e}vy strictly $\RFalpha$-stable distribution.
\end{proposition}
For $(\RFalpha,\RFtheta)\in\{(1,1),(1,-1)\}$, 
 the probability measure $\Green$ is a delta distribution,
 e.g. $G^1_1(x,t) = \delta_{x+t}$ and $G^1_{-1}(x,t) = \delta_{x-t}$,
 and is called trivial~\cite[Definition 13.6]{Sato:1999}.
However, we are interested in non-trivial probability measures $\Green$ for 
 \[ (\RFalpha,\RFtheta)\in\RFnontrivial:=\set{(\RFalpha,\RFtheta)\in\RFall}{ \abs{\RFtheta}<1} \ , \]
such that $\RFall = \RFnontrivial \cup \{(1,1), \ (1,-1)\}$.
Note, nonlocal Riesz-Feller $\RieszFeller$ operators are those with parameters
 \[ (\RFalpha,\RFtheta)\in\RFnonlocal:=\set{(\RFalpha,\RFtheta)\in\RFall}{0<\RFalpha< 2 \ , \quad  \abs{\RFtheta}<1}, \]
 such that $\RFnontrivial=\RFnonlocal \cup \{(2,0)\}$.
\begin{proposition}[{\cite[Lemma 2.1]{Achleitner+Kuehn:2015}}] \label{prop:SSPM} 
For $(\RFalpha,\RFtheta)\in\RFnontrivial$ the probability measure $\Green$ is absolutely continuous with respect to the Lebesgue measure
 and possesses a probability density which will be denoted again by $\Green$.
For all $(x,t)\in \R \times (0,\infty)$ 
 the following properties hold;
  \begin{enumerate}[label=(\alph*)] 
    \item \label{K:SFDE:prop0} $\Green(x,t)\geq 0$. If $\RFtheta\neq \pm \RFalpha$ then $\Green(x,t)> 0$;
    \item \label{K:SFDE:prop2} $\|\Green(\cdot,t)\|_{L^1(\R)}=1$;
    \item \label{K:SFDE:prop1} $\Green(x,t)=t^{-1/\RFalpha} \Green (x t^{-1/\RFalpha},1)$;
    \item \label{K:SFDE:prop3} $\Green(\cdot,s)\ast \Green(\cdot,t) = \Green(\cdot,s+t)$ for all $s,t\in(0,\infty)$;
    \item \label{K:SFDE:prop5} $\Green\in C^\infty_0(\R\times (0,\infty))$.
  \end{enumerate}
\end{proposition}
The L\'evy measure~$\nu$ of a Riesz-Feller operator $\RieszFeller$ 
 with $(\RFalpha,\RFtheta)\in\RFnonlocal$ is absolutely continuous with respect to Lebesgue measure and satisfies 
\begin{equation} \label{eq:LevyMeasure1D}
\nu(\dd[y]) = 
 \begin{cases}
  c_-(\RFtheta) y^{-1-\RFalpha}\ \dd[y] & \text{on } (0,\infty)\ , \\
  c_+(\RFtheta) |y|^{-1-\RFalpha}\ \dd[y] & \text{on } (-\infty,0)\ ,
 \end{cases}
\end{equation}
with $c_\pm(\RFtheta)= \Gamma(1+\RFalpha) \sin ( (\RFalpha \pm \RFtheta) \pi/2 ) /\pi$,
 see~\cite{Mainardi+Luchko+Pagnini:2001,Uchaikin:2013:bothVolumes}.

To study a TWP for evolution equations involving Riesz-Feller operators, 
 it is necessary to extend the Riesz-Feller operators to $C^2_b(\R)$.
Their singular integral representations~\eqref{op:Levy} may be used to accomplish this task.
\begin{theorem}[\cite{Achleitner+Kuehn:2015}] \label{thm:RieszFeller:extension}
If $(\RFalpha,\RFtheta)\in\RFnonlocal$ with $\RFalpha\ne 1$,
 then for all $f\in\SchwartzTF(\R)$ and $x\in\R$ 
\begin{align}
\RieszFeller f(x) = &\frac{c_+(\RFtheta) - c_-(\RFtheta)}{1-\RFalpha} f'(x) \nonumber \\
  &+ c_+(\RFtheta) \integrall{0}{\infty}{ \frac{f(x+y)-f(x)-f'(x)\,y 1_{(-1,1)}(y)}{y^{1+\RFalpha}} }{y} \label{eq:RieszFeller1} \\
  &+ c_-(\RFtheta) \integrall{0}{\infty}{ \frac{f(x-y)-f(x)+f'(x)\,y 1_{(-1,1)}(y)}{y^{1+\RFalpha}} }{y} \nonumber
\end{align}
with $c_\pm(\RFtheta)= \Gamma(1+\RFalpha) \sin ( (\RFalpha \pm \RFtheta) \pi/2 ) /\pi$.
Alternative representations are
\begin{itemize}
\item If $0<\RFalpha<1$, then 
 \[ 
 \RieszFeller f(x) = c_+(\RFtheta) \integrall{0}{\infty}{ \frac{f(x+y)-f(x)}{y^{1+\RFalpha}} }{y}
  + c_-(\RFtheta) \integrall{0}{\infty}{ \frac{f(x-y)-f(x)}{y^{1+\RFalpha}} }{y} \,.
 \] 
\item If $1<\RFalpha<2$, then 
 \begin{multline} \label{eq:RieszFeller1b}
 \RieszFeller f(x) = c_+(\RFtheta) \integrall{0}{\infty}{ \frac{f(x+y)-f(x)-f'(x)\,y}{y^{1+\RFalpha}} }{y} \\
  + c_-(\RFtheta) \integrall{0}{\infty}{ \frac{f(x-y)-f(x)+f'(x)\,y}{y^{1+\RFalpha}} }{y} \,.
 \end{multline}
\end{itemize}
\end{theorem}

These representations allow to extend Riesz-Feller operators~$\RieszFeller$ to $C^2_b(\R)$ such that $\RieszFeller C^2_b(\R)\subset C_b(\R)$.
For example, one can show
\begin{proposition}[{\cite[Proposition 2.4]{Achleitner+Kuehn:2015}}] \label{prop:RieszFeller:estimate}
For $(\RFalpha,\RFtheta)\in\RFall$ with $1<\RFalpha<2$,
 the integral representation~\eqref{eq:RieszFeller1b} of $\RieszFeller$ is well-defined for functions $f\in C^2_b(\R)$ with
 \begin{equation} \label{eq:estimate:RieszFeller}
  \sup_{x\in\R} |\RieszFeller f(x)|
      \leq \cK \|{f''}\|_{C_b(\R)} \frac{M^{2-\RFalpha}}{2-\RFalpha}
      + 4 \cK \|f'\|_{C_b(\R)} \frac{M^{1-\RFalpha}}{\RFalpha-1} < \infty
 \end{equation}
 for some positive constants $M$
 and $\cK=\frac{\Gamma(1+\RFalpha)}{\pi} \abs{ \sin ( (\RFalpha+\RFtheta) \frac{\pi}{2} ) + \sin( (\RFalpha-\RFtheta) \frac{\pi}{2} ) }$.

\end{proposition}
Estimate~\eqref{eq:estimate:RieszFeller} is a key estimate,
 which is used to adapt Chen's approach~\cite{Chen:1997} to the TWP for nonlocal reaction-diffusion equations with Riesz-Feller operators~\cite{Achleitner+Kuehn:2015}. 
\medskip

\section{TWP for classical evolution equations} \label{sec:local}

In this section we review the importance of the TWP for reaction-diffusion equations and scalar conservation laws with higher-order regularizations, respectively.
\subsection{reaction-diffusion equations}
A scalar reaction-diffusion equations is a partial differential equation
 \begin{equation} \label{RD:r:local}
   \sdiff{u}{t} = r(u) + \sigma \sdifff{u}{x}{2} \ , \quad t>0 \ ,\quad x\in\R \ ,
 \end{equation}
 for some positive constant~$\sigma>0$, as well as a nonlinear function $r:\R\to\R$ and second-order derivative $\sdifff{u}{x}{2}$
 modeling reaction and diffusion, respectively.
The TWP for given endstates $\upm$
 is to study the existence of a TWS~$(\profile,\RDspeed)$ for~\eqref{RD:r:local} in the sense of Definition~\ref{def:TWS}.
If the profile $\profile\in C^2(\R)$ is bounded, then it satisfies $\lim_{\xi\to\pm\infty} \profile^{(n)}(\xi) =0$ for $n=1,2$.
A TWS~$\TWS$ satisfies the TWE
 \begin{equation}\label{RD:TWE:local}
  -c \profile' = r(\profile) + \sigma \profile'' \ , \quad \xi\in\R \ . 
 \end{equation}

\emph{phase plane analysis.}
A traveling wave profile~$\profile$ is a heteroclinic orbit of the TWE~\eqref{RD:TWE:local} connecting the endstates~$\upm$.
To identify necessary conditions on the existence of TWS,
 TWE~\eqref{RD:TWE:local} is written as a system of first-order ODEs for $u$, $v:=u'$:
 \begin{equation} \label{RD:TWE:local:system}
  \Diff{}{\xi} \binom{u}{v} = \binom{v}{(-r(u) -\RDspeed v)/\sigma} =: F(u,v) \ , \quad \xi\in\R \ . 
 \end{equation}
First, an endstate $(u_s,v_s)$ of a heteroclinic orbit has to be a stationary state of $F$, i.e. $F(u_s,v_s)=0$,
 which implies $v_s\equiv 0$ and $r(u_s) =0$.
Second, $(\um,0)$ has to be an unstable stationary state of~\eqref{RD:TWE:local:system}
 and $(\up,0)$ either a saddle or a stable node of~\eqref{RD:TWE:local:system}.
As long as a stationary state $(u_s,v_s)$ is hyperbolic,
 i.e. the linearization of $F$ at $(u_s,v_s)$ has only eigenvalues $\lambda$ with non-zero real part,
 the stability of $(u_s,v_s)$ is determined by these eigenvalues.
The linearization of $F$ at $(u_s,v_s)$ is 
 \begin{equation}
   \DD F(u_s,v_s) = \begin{pmatrix} 0 & 1 \\ -r'(u_s)/\sigma & -\RDspeed/\sigma \end{pmatrix} \ . 
 \end{equation}
Eigenvalues~$\lambda_\pm$ of the Jacobian $\DD F(u_s,v_s)$ satisfy the characteristic equation $\lambda^2 + \lambda\RDspeed/\sigma +r'(u_s)/\sigma=0$.
Moreover, $\lambda_- + \lambda_+ = -\RDspeed/\sigma$ and $\lambda_- \lambda_+ = r'(u_s)/\sigma$.
The eigenvalues~$\lambda_\pm$ of the Jacobian $\DD F(u_s,v_s)$ are
 \begin{equation} \label{RD:DF:EV}
    \lambda_\pm = -\frac{\RDspeed}{2\sigma} \pm \sqrt{ \frac{\RDspeed^2}{4\sigma^2} - \frac{r'(u_s)}{\sigma}} 
                = \frac{-\RDspeed \pm \sqrt{ \RDspeed^2 - 4\sigma r'(u_s)}}{2\sigma} \ .
 \end{equation}
Thus $r'(u_s)<0$ ensures that $(u_s,0)$ is a saddle point,
 i.e. with one positive and one negative eigenvalue.

\emph{balance of potential.} 
The potential $R$ (of the reaction term $r$) is defined as $R(u) :=\integrall{0}{u}{ r(\upsilon)}{\upsilon}$. 
The potentials of the endstates $\upm$ are  called \emph{balanced} if $R(\up)=R(\um)$ and \emph{unbalanced} otherwise.
A formal computation reveals a connection between the sign of $\RDspeed$ and the balance of the potential $R(u)$:
Multiplying TWE~\eqref{RD:TWE:local} with $\profile'$, integrating on $\R$ and using~\eqref{TWS:limits}, yields
\begin{equation} \label{eq:c+R}
   -\RDspeed \|\profile'\|_{L^2}^2
     = \integrall{\um}{\up}{ r(\upsilon)}{\upsilon} = R(\up) -R(\um) \ , 
\end{equation}
since $\integral{\R}{ \profile'' \profile'}{\xi} = 0$ due to~\eqref{TWS:limits}. 
Thus $-\sgn \RDspeed = \sgn (R(\up) -R(\um))$.
In case of a balanced potential the wave speed $\RDspeed$ is zero, hence the TWS is stationary.

\begin{definition} \label{def:reaction:type}
Assume $\um>\up$. A function $r\in C^1(\R)$ with $r(\upm)=0$ is  
\begin{itemize}
 \item \emph{monostable} if $r'(\um)<0$, $r'(\up)>0$ and $r(u)>0$ for $u\in(\up,\um)$. 
 \item \emph{bistable} if $r'(\upm)<0$ and
   \[ \exists \uStar\in(\up,\um) \ : \ r(u) 
	       \begin{cases}
				    <0 & \text{for } u\in (\up,\uStar) \ , \\
						>0 & \text{for } u\in (\uStar,\um) \ .
				 \end{cases}
	 \]
\item \emph{unstable} if $r'(\upm)>0$.
\end{itemize}
\end{definition}
We chose a very narrow definition compared to \cite{Volpert3:1994}.
Moreover, in most applications of reaction-diffusion equations a quantity $u$ models a density of a substance/population. 
In these situations only nonnegative states $\upm$ and functions $u$ are of interest.
\begin{proposition}[{\cite[\S 2.2]{Volpert3:1994}}] \label{prop:RD:TWS:existence}
Assume $\sigma>0$ and $\um>\up$.
\begin{itemize}
\item If $r$ is monostable, then there exists a positive constant $\RDspeed_*$
 such that for all $\RDspeed\geq \RDspeed_*$ there exists a monotone TWS $(\profile,\RDspeed)$ of \eqref{RD:r:local}
 in the sense of Definition~\ref{def:TWS}.
 For $\RDspeed<\RDspeed_*$ no such monotone TWS exists (however oscillatory TWS may exist).
\item If $r$ is bistable, then there exists an (up to translations) unique monotone TWS $(\profile,\RDspeed)$ of \eqref{RD:r:local}
 in the sense of Definition~\ref{def:TWS}.
\item If $r$ is unstable, then there does not exist a monotone TWS $(\profile,\RDspeed)$ of \eqref{RD:r:local}.
\end{itemize}
\end{proposition}
If a TWS~$(\profile,\RDspeed)$ exists,
 then a closer inspection of the eigenvalues~\eqref{RD:DF:EV} at $(\up,0)$ indicates the geometry of the profile~$\profile$ for large $\xi$:
\[ \RDspeed^2 - 4\sigma r'(\up)
      \begin{cases}
			  \geq 0 & \text{TWS with monotone decreasing profile~$\profile$ for large $\xi$}; \\
	      < 0 & \text{TWS with oscillating profile~$\profile$ for large $\xi$}.
      \end{cases}
\]

\subsection{Korteweg-de Vries-Burgers equation (KdVB)} 
 A generalized KdVB equation is a scalar partial differential equation
  \begin{equation} \label{eq:KdVB:f}
   \sdiff{u}{t} + \sdiff{}{x} f(u) = \epsilon \sdifff{u}{x}{2} + \delta \sdifff{u}{x}{3}, \quad x\in\R, \quad t>0, 
  \end{equation}
  for some flux function $f:\R\to\R$ as well as constants $\epsilon>0$ and $\delta\in\R$.
 The TWP for given endstates $\upm$
  is to study the existence of a TWS~$\TWS$ for~\eqref{eq:KdVB:f} in the sense of Definition~\ref{def:TWS}.
 The importance of the TWP for KdVB equations in the shock wave theory of (scalar) hyperbolic conservation laws is discussed in Section~\ref{sec:ShockWaveTheory}.
 A TWS~$\TWS$ satisfies the TWE
 \begin{equation}\label{KdVB:TWE:local}
  -c \profile' + f'(\profile)\ \profile' = \epsilon \profile'' + \delta \profile''' \ , \quad \xi\in\R \ , 
 \end{equation}
or integrating on $(-\infty,\xi]$ and using~\eqref{TWS:limits},
 \begin{equation}\label{KdVB:TWE:local:integrated}
  h(\profile) := f(\profile) -c \profile -(f(\um) -c \ \um) = \epsilon \profile' + \delta \profile'' \ , \quad \xi\in\R \ . 
 \end{equation}

\emph{connection with reaction-diffusion equation.}
A TWS $u(x,t)=\profile(x-c t)$ of a generalized Korteweg-de Vries-Burgers equation~\eqref{eq:KdVB:f}
satisfies TWE~\eqref{KdVB:TWE:local:integrated}.
Thus $v(x,t)=\profile(x-\epsilon t)$ is a TWS~$(\profile,\epsilon)$ of the reaction-diffusion equation
\begin{equation} \label{eq:RD:h}
 \sdiff{v}{t} = -h(v) + \delta \sdifff{v}{x}{2}\ , \quad x\in\R\ , \quad t>0\ . 
\end{equation}
	
\emph{phase plane analysis.}
Following the analysis of TWE~\eqref{RD:TWE:local}
 for a reaction-diffusion equation~\eqref{RD:r:local} with $r(u)=-h(u)$ and $\sigma=\delta$,
 necessary conditions on the parameters can be identified.
First, a TWE is rewritten as a system of first-order ODEs with vector field~$F$.
Then the condition on stationary states implies that endstates~$\upm$ and wave speed~$c$ have to satisfy
 \begin{equation}\label{cond:up:stationary}
  f(\up)-f(\um) = \speed (\up - \um) \ .
 \end{equation}
This condition is known in shock wave theory as Rankine-Hugoniot condition~\eqref{RH} on the shock triple $\ShockTriple$.
The (nonlinear) stability of hyperbolic stationary states~$(u_s,v_s)$ of~$F$ is determined by
 the eigenvalues
 \begin{equation} \label{KdVB:DF:EV}
    \lambda_\pm
      = -\frac12 \frac\epsilon\delta \pm \frac{\sqrt{ \epsilon^2 + 4 \delta h'(u_s)}}{2 |\delta|} 
 \end{equation}
of the Jacobian $\DD F(u_s,v_s)$.
If $\epsilon,\delta>0$, then $(\up,0)$ is always either a saddle or stable node,
 and $h'(\um)=f'(\um)-\speed>0$ ensures that $(\um,0)$ is unstable.
For example, Lax' entropy condition~\eqref{cond:entropy:Lax}, 
 i.e. $f'(\up)<\speed<f'(\um)$, implies the latter condition.

\medskip \noindent
\textbf{convex flux functions.}
\begin{theorem} \label{thm:KdVB:TWP:f:convex}
Suppose $f\in C^2(\R)$ is a strictly convex function. 
Let $\epsilon$, $\delta$ be positive and
 let $\ShockTriple$ satisfy the Rankine-Hugoniot condition~\eqref{RH} and
 the entropy condition~\eqref{cond:entropy:Lax}, i.e. $\um>\up$.
Then, there exists an (up to translations) unique TWS $(\profile,c)$ of~\eqref{eq:KdVB:f}
 in the sense of Definition~\ref{def:TWS}.
\end{theorem}

\begin{proof}
 We consider the associated reaction-diffusion equation~\eqref{eq:RD:h},
 i.e. $\sdiff{u}{t} = r(u) +\delta\sdifff{u}{x}{2}$ with $r(u)=-h(u)$.
 Due to~\eqref{RH} and~\eqref{cond:entropy:Lax},
  $r(u)$ is monostable in the sense of Definition~\ref{def:reaction:type}.
 Moreover, function $r$ is strictly concave,
  since $r''(u)=-f''(u)$ and $f\in C^2(\R)$ is strictly convex. 
 In fact, $(\upm,0)$ are the only stationary points of system~\eqref{RD:TWE:local:system}, where $(\um,0)$ is a saddle point and $(\up,0)$ is a stable node.
Thus, for all wave speeds $\epsilon$ there exists a TWS $(\profile,\epsilon)$ -- with possibly oscillatory profile~$\profile$ -- of
 reaction-diffusion equation~\eqref{eq:RD:h}.
Moreover, $(\profile,c)$ is a TWS of~\eqref{eq:KdVB:f},
 due to~\eqref{KdVB:TWE:local}--\eqref{eq:RD:h}.
\qed
\end{proof}

The TWP for KdVB equations~\eqref{eq:KdVB:f} with Burgers' flux~$f(u) =u^2$ has been investigated in~\cite{Bona+Schonbek:1985}.
The sign of $\delta$ in~\eqref{eq:KdVB:f} is irrelevant,
 since it can be changed by a transformation
 $\tilde{x}=-x$ and $\tilde{u}(\tilde{x},t)=-u(x,t)$,
 see also~\cite{Jacobs+McKinney+Shearer:1995}.
First, the results in Theorem~\ref{thm:KdVB:TWP:f:convex} on the existence of TWS and geometry of its profiles are proven. 
More importantly, the authors investigate the convergence of profiles $\profile(\xi;\epsilon,\delta)$ in the limits $\epsilon\to 0$, $\delta\to 0$, as well as $\epsilon$ and $\delta$ tending to zero simultaneously. 
Assuming that the ratio $\delta/\epsilon^2$ remains bounded,
 they show that the TWS converge to the classical Lax shocks
 for this vanishing diffusive-dispersive regularization~\cite{Bona+Schonbek:1985}.

\medskip \noindent
\textbf{concave-convex flux functions.}
\begin{definition}[\cite{LeFloch:2002}] \label{flux:concave-convex}
A function $f\in C^3(\R)$ is called \emph{concave-convex} if 
\begin{equation}\label{cond:flux:concave-convex}
 u f''(u)> 0 \quad \forall u\ne 0 \ , \quad 
 f'''(0)\ne 0 \ , \quad
 \lim_{u\to\pm\infty} f'(u) = +\infty \ .
\end{equation}
\end{definition}
Here the single inflection point is shifted without loss of generality to the origin.
We consider a cubic flux function $f(u)=u^3$ as the prototypical concave-convex flux function with a single inflection point,
 see \cite{Hayes+Shearer:1999, LeFloch:2002}.

\begin{proposition}[{\cite{Jacobs+McKinney+Shearer:1995,Hayes+LeFloch:1997}}] \label{prop:JMcKS}
Suppose $f(u) =u^3$ and $\epsilon>0$.
\begin{enumerate}[label=(\alph*)]
\item \label{JMcKS:delta:non-pos} If $\delta\leq 0$ then a TWS $(\profile,c)$ of \eqref{eq:KdVB:f} exists 
 if and only if $\ShockTriple$ satisfy the Rankine-Hugoniot condition~\eqref{RH} and the entropy condition~\eqref{cond:entropy:Lax}.
\item \label{JMcKS:delta:pos} If $\delta> 0$ then a TWS $(\profile,c)$ of \eqref{eq:KdVB:f} exists for $\um>0$ if and only if $\up\in S(\um)$ with
\begin{equation} \label{set:Sum:JMcKS}
S(\um) = 
\begin{cases}
[-\frac{\um}{2} , \um ) & \xx{if} \um \leq 2\beta \ , \\
\{ -\um + \beta \} \cup [-\beta , \um ) & \xx{if} \um > 2\beta \ ,
\end{cases} 
\end{equation}
where the coefficient $\beta$ is given by $\beta=\tfrac{\sqrt{2}}{3}\tfrac{\epsilon}{\sqrt{\delta}}$.
\end{enumerate}
\end{proposition}
\begin{proof}
Following the discussion from \eqref{eq:KdVB:f}--\eqref{eq:RD:h},
 we consider the associated reaction-diffusion equation~\eqref{eq:RD:h},
 i.e. $\sdiff{u}{t} = r(u) + \delta \sdifff{u}{x}{2}$ with $r(u)=-h(u)$.
From this point of view,
  we need to classify the reaction term $r(u)=-h(u)$:
Whereas $r(\um)=0$ by definition,
 $r(\up)=0$ if and only if $\ShockTriple$ satisfies the Rankine-Hugoniot condition~\eqref{RH}.
 The Rankine-Hugoniot condition implies $\speed=\up^2+\up\ \um+\um^2$.
 Hence, the reaction term $r(u)$ has a factorization
 \begin{equation} \label{r:cubic}
    r(u) = -(u^3-\um^3-c(u-\um)) 
	       = -(u-\um)\ (u-\up)\ (u+\up+\um)
 \end{equation}
Thus, $r(u)$ is a cubic polynomial with three roots $u_1 \leq u_2 \leq u_3$,
 such that $r(u) = -(u-u_1) (u-u_2) (u-u_3)$. 
In case of distinct roots $u_1 < u_2 < u_3$ we deduce $r'(u_1)<0$, $r'(u_2)>0$ and $r'(u_3)<0$.
The ordering of the roots $\upm$ and $\uStar=-\um-\up$ depending on $\upm$ is visualized in Figure~\ref{fig:classification:cubic}.
\begin{figure} 
\begin{center}
 \includegraphics{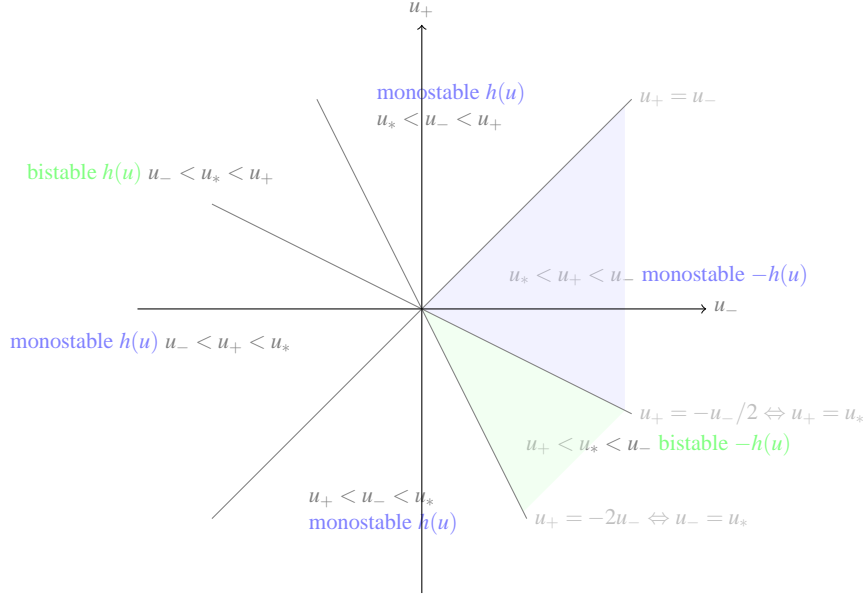}
\end{center}
 \caption{classification of the cubic reaction function $r(u)=-h(u)$ in~\eqref{r:cubic} depending on its roots $\um$, $\up$ and $\uStar=-\um-\up$ according to Definition~\ref{def:reaction:type}.}
 \label{fig:classification:cubic}
\end{figure}
 Next, we will discuss the results in Proposition~\ref{prop:JMcKS}\ref{JMcKS:delta:pos} (for $\um>0$ and $\delta>0$)
 via results on the existence of TWS for a reaction-diffusion equation~\eqref{eq:RD:h}.
 \begin{enumerate}
  \item For $\up<\uStar<\um$, function $r(u)$ is bistable, see also Figure~\ref{fig:classification:cubic}.
   Due to Proposition~\ref{prop:RD:TWS:existence},
	  there exists an (up to translations) unique TWS $(\profile,\epsilon)$
		with possibly negative wave speed.
	 Under our assumption that the wave speed $\epsilon$ is positive,
	  relation~\eqref{eq:c+R} yields the restriction $-\up>\um$. 
  In fact, for $\um> 2\beta$ and $\up=-\um+\beta$
	 there exists a TWS $(\profile,\epsilon)$ for reaction-diffusion equation~\eqref{eq:RD:h},
	 see \cite[Theorem 3.4]{Jacobs+McKinney+Shearer:1995}.
	The function $r$ is bistable with $\uStar=-\um-\up=-\beta$, hence $f'(\upm)>\speed$.
	This violates Lax' entropy condition~\eqref{cond:entropy:Lax}
   and is known in the shock wave theory as a slow undercompressive shock~\cite{LeFloch:2002}.
 \item For $\uStar<\up<\um$, function $r(u)$ is monostable, see Figure~\ref{fig:classification:cubic}.
  Due to Proposition~\ref{prop:RD:TWS:existence}, 
	  there exists a critical wave speed $\speed_*$,
    such that monotone TWS $(\profile,\epsilon)$ for~\eqref{eq:RD:h} exist for all $\epsilon\geq\speed_*$.
  However, not all endstates $(\um,\up)$ in the subset defined by $\uStar<\up<\um$ admit a TWS~$\TWS$,
	  see~\eqref{set:Sum:JMcKS} and Figure~3b).
	 The TWS $\TWS$ associated to non-classical shocks appear again, with reversed roles for the roots $\up$ and $\uStar$:
	 For $\um> 2\beta$ and $\up = -\beta$,
    there exists a TWS~$(\profile,\epsilon)$ for reaction-diffusion equation~\eqref{eq:RD:h},
		see \cite[Theorem 3.4]{Jacobs+McKinney+Shearer:1995}.
	 These TWS form a horizontal halfline in Figure~3b)
	  and divides the set defined by $\uStar<\up<\um$ into two subsets.
	 In particular, TWS exist only for endstates $(\um,\up)$ in the subset above this halfline.
  \item For $\up<\um<\uStar$, function $r(u)=-h(u)$ satisfies $r(u)<0$ for all $u\in(\up,\um)$,
	 see also Figure~\ref{fig:classification:cubic}.
   Thus the necessary condition~\eqref{eq:c+R} can not be fulfilled for positive $\RDspeed=\epsilon$,
	  hence there exists no TWS $(\profile,\epsilon)$ for the reaction-diffusion equation.
 \item For $\uStar<\um<\up$, function $r(u)$ is monostable with reversed roles of the endstates~$\upm$, see Figure~\ref{fig:classification:cubic}. Due to Proposition~\ref{prop:RD:TWS:existence}, there exists a TWS $(\profile,\epsilon)$ however satisfying $\lim_{\xi\to\mp\infty} \profile(\xi) = \upm$.
 \end{enumerate}
If $\delta=0$, then equation~\eqref{eq:KdVB:f} is a viscous conservation law,
 and its TWE \eqref{KdVB:TWE:local:integrated} is a simple ODE $-\epsilon \profile' = r(\profile)$ with $r(u)=-h(u)$.
Thus a heteroclinic orbit exists only for monostable $r(u)$,
 i.e. if the unstable node $\um$ and the stable node $\up$ are not separated by any other root of $r$.

If $\delta<0$, then we rewrite TWE~\eqref{KdVB:TWE:local:integrated} as $\epsilon \profile' = h(u) + |\delta| \profile''$. 
It is associated to a reaction-diffusion equation $\sdiff{u}{t} = h(u) + |\delta|\sdifff{u}{x}{2}$
 via a TWS ansatz $u(x,t) = \profile(x-(-\epsilon) t)$;
 note the change of sign for the wave speed. 
If $\up<\uStar<\um$ then $h(u)$ is an unstable reaction function. 
Thus there exists no TWS $(\profile,-\epsilon)$ according to Proposition~\ref{prop:RD:TWS:existence}. 
If $\uStar<\up<\um$ then function $h(u)=-r(u)$ satisfies $h(u)<0$ for all $u\in(\up,\um)$,
 see also Figure~\ref{fig:classification:cubic}. 
The necessary condition~\eqref{eq:c+R} is still fine,
 since also the sign of the wave speed changed. 
In contrast to the case $\delta>0$, there exists no TWS connecting $\um$ with $\uStar$, which would indicate a bifurcation. 
Thus, the existence of TWS for all pairs $(\um,\up)$ in the subset defined by $\uStar<\up<\um$ can be proven.
The TWP for other pairs $(\um,\up)$ is discussed similarly.

\qed
\end{proof}
 
\begin{figure}
\begin{center}
\includegraphics{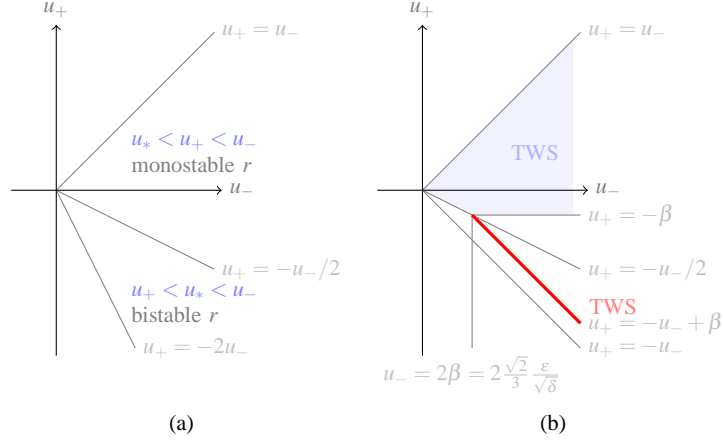}
\end{center}
\caption{a) classification of reaction function $r$ depending on its roots $\um$, $\up$ and $\uStar=-\um-\up$;
b) Endstates $\upm$ in the shaded region and on the thick line can be connected by TWS of the cubic KdVB equation;
 TWS in the shaded region and on the thick line are associated to classical and non-classical shocks of $\sdiff{u}{t}+\sdiff{u^3}{x} = 0$, respectively. 
For a classical shock the shock triple satisfies Lax' entropy condition $f'(\um)>\speed>f'(\up)$;
 i.e. characteristics in the Riemann problem meet at the shock.
In contrast, the non-classical shocks are of slow undercompressive type,
  i.e. characteristics in the Riemann problem cross the shock.
}
 \label{fig:KdVB}
\end{figure}

\section{TWP for nonlocal evolution equations} \label{sec:nonlocal}

\subsection{reaction-diffusion equations}
The first example of a reaction-diffusion equation with nonlocal diffusion is the integro-differential equation
\begin{equation} \label{RD:r:J*u-u}
 \sdiff{u}{t} = J \ast u - u + r(u) \ , \quad t>0 \ , \quad x\in\R \ ,
\end{equation}
for some even, non-negative function $J$ with mass one, i.e. for all $x\in\R$
\begin{equation} \label{as:J} 
 J\in C(\R)\ , \quad J\geq 0\ ,  \quad J(x)=J(-x)\ , \quad \integral{\R}{ J(y)}{y}=1 \ , 
\end{equation}
and some function~$r$.
The operator $\Levy[u] = J\ast u -u$ is a L\'evy operator, see~\eqref{IG:compoundPoisson:II},
 which models nonlocal diffusion. 
It is the infinitesimal generator of a compound Poisson stochastic process, which is a pure jump process.

The TWP for given endstates $\upm$
 is to study the existence of a TWS~$\TWS$ for~\eqref{RD:r:J*u-u} in the sense of Definition~\ref{def:TWS}.
Such a TWS $(\profile,\RDspeed)$ satisfies the TWE
 $-\RDspeed \profile' = J \ast \profile - \profile + r(\profile)$ for $\xi\in\R$.
Next, we recall some results on the TWP for~\eqref{RD:r:J*u-u}, 
 which will depend crucially on the type of reaction function $r$ 
 and the tail behavior of a kernel function $J$.
We will present the existence of TWS with monotone decreasing profiles $\profile$, 
 which will follow from the cited literature after a suitable transformation. 
\begin{proposition}[{(monostable \cite{Coville+Dupaigne:2007}), (bistable \cite{Bates+etal:1997,Chen:1997})}] \label{prop:RD:J*u-u:TWS:existence}
 Suppose $\um>\up$ and consider reaction functions~$r$ in the sense of Definition~\ref{def:reaction:type}.
 Suppose $J\in W^{1,1}(\R)$ and its continuous representative satisfies~\eqref{as:J}.
\begin{itemize}
\item If $r$ is monostable and there exists $\lambda>0$ such that $\integral{\R}{J(y)\ \exp (\lambda y)}{y} < \infty$
 then there exists a positive constant $\RDspeed_*$ such that for all $\RDspeed\geq \RDspeed_*$
 there exists a monotone TWS $(\profile,\RDspeed)$ of \eqref{RD:r:J*u-u}.
 For $\RDspeed<\RDspeed_*$ no such monotone TWS exists.
\item If $r$ is bistable and $\integral{\R}{\abs{y} J(y)}{y}<\infty$,
 then there exists an (up to translations) unique monotone TWS $(\profile,\RDspeed)$ of \eqref{RD:r:J*u-u}.
\end{itemize}
\end{proposition}
For monostable reaction functions, 
 the tail behavior of kernel function $J$ is very important.
There exist kernel functions $J$, such that TWS exist only for bistable -- but not for monostable -- reaction functions $r$,
 see~\cite{Yagisita:2009:monostable}.
The prime example are kernel functions $J$ 
 which decay more slowly than any exponentially decaying function as $|x|\to\infty$ in the sense that
 $J(x)\exp(\eta |x|) \to\infty$ as $|x|\to\infty$ for all $\eta>0$.

For reaction-diffusion equations of bistable type,
 Chen established a unified approach~\cite{Chen:1997} to prove the existence, uniqueness
 and asymptotic stability with exponential decay of traveling wave solutions.
The results are established for a subclass of nonlinear nonlocal evolution equations
 \[ \sdiff{u}{t} (x,t) = \operator{A}[u(\cdot,t)](x) \Xx{for} (x,t)\in\R\times (0,T]\ , \]
 where the nonlinear operator $\operator{A}$ is assumed to 
 \begin{enumerate}[label=(\alph*)]
  \item be independent of $t$;
  \item generate a $L^\infty$ semigroup;
  \item be translational invariant, i.e. $\operator{A}$ satisfies for all $u\in\dom\operator{A}$ the identity
   \[ \operator{A}[u(\cdot+h)](x) = \operator{A}[u(\cdot)](x+h)\quad \forall x\ ,h \in\R\,. \]
   Consequently, there exists a function $r:\R\to\R$
   which is defined by $\operator{A}[\upsilon \mathbf 1]=r(\upsilon)\mathbf 1$
	 for $\upsilon\in\R$ and the constant function $\mathbf 1:\R\to\R$, $x\mapsto 1$.
  This function $r$ is assumed to be bistable in the sense of Definition~\ref{def:reaction:type}; 
  \item satisfy a comparison principle:
    If $\sdiff{u}{t}\geq\operator{A}[u]$, $\sdiff{v}{t}\leq\operator{A}[v]$ and $u(\cdot ,0)\gneq v(\cdot ,0)$,
    then $u(\cdot ,t)>v(\cdot ,t)$ for all $t>0$.
 \end{enumerate}
Chen's approach relies on the comparison principle
 and the construction of sub- and supersolutions for any given traveling wave solution.
Importantly, the method does not depend on the balance of the potential. 
More quantitative versions of the assumptions on $\operator{A}$ are needed in the proofs.
Finally integro-differential evolution equations
\begin{equation} \label{eq:Chen:IDE}
 \sdiff{u}{t} = \epsilon \sdifff{u}{x}{2} + G(u,J_1\ast S^1(u),\ldots,J_n\ast S^n(u))
\end{equation}
 are considered for some diffusion constant $\epsilon\geq 0$, smooth functions $G$ and $S^k$,
 and kernel functions $J_k\in C^1(\R)\cap W^{1,1}(\R)$ satisfying~\eqref{as:J} where $k=1,\ldots,n$. 
Additional assumptions on the model parameters guarantee that 
 an equation~\eqref{eq:Chen:IDE} can be interpreted as a reaction-diffusion equation with bistable reaction function
 including equations~\eqref{RD:r:local} and~\eqref{RD:r:J*u-u} as special cases.

Another example of reaction-diffusion equations with nonlocal diffusion are the integro-differential equations
\begin{equation} \label{RD:r:RieszFeller}
  \sdiff{u}{t} = \RieszFeller u + r(u) \ , \quad t>0 \ , \quad x\in\R \ ,
\end{equation}
for a (particle) density $u=u(x,t)$, some function $r=r(u)$,
and a Riesz-Feller operator $\RieszFeller$ with $(\alpha,\theta)\in\RFall$.
The nonlocal Riesz-Feller operators are models for superdiffusion,
 where from a probabilistic view point a cloud of particle is assumed to spread faster than by following Brownian motion.
Integro-differential equation~\eqref{RD:r:RieszFeller} can be derived as a macroscopic equation for a particle density in the limit of modified Continuous Time Random Walk (CTRW), see~\cite{Mendez+etal:2010}.
In the applied sciences, equation~\eqref{RD:r:RieszFeller} has found many applications, 
 see \cite{Uchaikin:2013:bothVolumes, Volpert+etal:2013} for extensive reviews on modeling, formal analysis and numerical simulations. 

The TWP for given endstates $\upm$
 is to study the existence of a TWS~$\TWS$ for~\eqref{RD:r:RieszFeller} in the sense of Definition~\ref{def:TWS}.
Such a TWS $(\profile,\RDspeed)$ satisfies the TWE
\begin{equation} \label{RD:TWE:RieszFeller}
 -\RDspeed \profile' = \RieszFeller \profile + r(\profile) \ , \quad \xi\in\R \ .
\end{equation}
First we collect mathematical rigorous results about the TWP associated to~\eqref{RD:r:RieszFeller}
 in case of the fractional Laplacian $\RieszFellerII{\RFalpha}{0} = -(-\Delta)^{\RFalpha/2}$ for $\RFalpha\in(0,2)$,
 i.e. a Riesz-Feller operator $\RieszFeller$ with $\RFtheta=0$.
\begin{proposition}[{(monostable \cite{Cabre+Roquejoffre:2009,Cabre+Roquejoffre:2013,Engler:2010}), (bistable \cite{Cabre+SolaMorales:2005, Cabre+Sire:2014, Cabre+Sire:2015, Palatucci+etal:2013,Chmaj:2013,Gui+Zhao:2015})}] \label{prop:RD:fracLaplacian:TWS:existence}
 Suppose $\um>\up$. Consider the TWP for reaction-diffusion equation~\eqref{RD:r:RieszFeller} with functions~$r$ in the sense of Definition~\ref{def:reaction:type} and fractional Laplacian~$\RieszFellerII{\RFalpha}{0}$, i.e. symmetric Riesz-Feller operators $\RieszFeller$ with $0<\RFalpha<2$ and $\RFtheta=0$.
\begin{itemize}
\item If $r$ is monostable then there does not exist any TWS $(\profile,\RDspeed)$ of \eqref{RD:r:RieszFeller}.
\item If $r$ is bistable then there exists an (up to translations) unique monotone TWS $(\profile,\RDspeed)$ of \eqref{RD:r:RieszFeller}.
\end{itemize}
\end{proposition}
For monostable reaction functions,
 Cabr\'e and Roquejoffre prove that a front moves exponentially in time~\cite{Cabre+Roquejoffre:2009,Cabre+Roquejoffre:2013}.
They note that the genuine algebraic decay of the heat kernels~$\GreenII{\RFalpha}{0}$ associated to fractional Laplacians
 is essential to prove the result, which implies that no TWS with constant wave speed can exist.
Engler~\cite{Engler:2010} considered the TWP for~\eqref{RD:r:RieszFeller}
 for a different class of monostable reaction functions $r$ and non-extremal Riesz-Feller operators~$\RieszFeller$
 with $(\RFalpha,\RFtheta)\in\RFnonextremal$ and $\RFnonextremal := \set{(\RFalpha,\RFtheta)\in\RFall}{ |\RFtheta|<\min\{\RFalpha, 2-\RFalpha\} }$.
Again the associated heat kernels $\Green(x,t)$ with $(\RFalpha,\RFtheta)\in\RFnonextremal$ decay algebraically in the limits $x\to\pm\infty$, see~\cite{Mainardi+Luchko+Pagnini:2001}.

To our knowledge,
 we established the first result~\cite{Achleitner+Kuehn:2015}
 on existence, uniqueness (up to translations) and stability of traveling wave solutions of~\eqref{RD:r:RieszFeller}
 with Riesz-Feller operators~$\RieszFeller$ for $(\RFalpha,\RFtheta)\in\RFall$ with $1<\RFalpha< 2$
 and bistable functions~$r$. 
We present our results for monotone decreasing profiles,
 which can be inferred from our original result after a suitable transformation.
\begin{theorem}[\cite{Achleitner+Kuehn:2015}] \label{thm:AK:RD:RieszFeller}
Suppose $\um>\up$, $(\RFalpha,\RFtheta)\in\RFall$ with $1<\RFalpha< 2$, and $r\in C^\infty(\R)$ is a bistable reaction function. 
Then there exists an (up to translations) unique monotone decreasing TWS $\TWS$ of~\eqref{RD:r:RieszFeller} in the sense of Definition~\ref{def:TWS}.
\end{theorem} 

The technical details of the proof are 
contained in~\cite{Achleitner+Kuehn:2015}, whereas in~\cite{Achleitner+Kuehn:2015:BCAM} we give a concise overview of the proof strategy
and visualize the results also numerically. 
In a forthcoming article~\cite{Achleitner+Kuehn:201x},
 we extend the results to all non-trivial Riesz-Feller operators $\RieszFeller$ with $(\RFalpha,\RFtheta)\in\RFnontrivial$. 
The smoothness assumption on $r$ is convenient, but not essential. 
To prove Theorem~\ref{thm:AK:RD:RieszFeller}, 
 we follow -- up to some modifications -- the approach of Chen~\cite{Chen:1997}.
It relies on a strict comparison principle
 and the construction of sub- and supersolutions for any given TWS.
His quantitative assumptions on operator~$\operator{A}$ are too strict, 
 such that his results are not directly applicable.
A modification allows to cover the TWP for~\eqref{RD:r:RieszFeller}
 for all Riesz-Feller operators~$\RieszFeller$ with $1<\RFalpha< 2$ also for non-zero $\RFtheta$,
 and all bistable functions~$r$ regardless of the balance of the potential.

Next, we quickly review different methods to study the TWP
 of reaction-diffusion equations~\eqref{RD:r:RieszFeller} with bistable function~$r$ and fractional Laplacian.
In case of a classical reaction-diffusion equation~\eqref{RD:r:local},
 the existence of a TWS can be studied via phase-plane analysis~\cite{Aronson+Weinberger:1974, Fife+McLeod:1977}.
This method has no obvious generalization to our TWP for~\eqref{RD:r:RieszFeller},
 since its traveling wave equation~\eqref{RD:TWE:RieszFeller} is an integro-differential equation.
The variational approach has been focused -- so far -- on symmetric diffusion operators such as fractional Laplacians 
 and on balanced potentials, hence covering only stationary traveling waves~\cite{Palatucci+etal:2013}.
Independently, the same results are achieved in~\cite{Cabre+SolaMorales:2005, Cabre+Sire:2014, Cabre+Sire:2015}
by relating the stationary TWE~\eqref{RD:TWE:RieszFeller}$_{\RFtheta=0, \RDspeed=0}$ 
 via~\cite{Caffarelli+Silvestre:2007} to a boundary value problem for a nonlinear partial differential equation.
The homotopy to a simpler TWP has been used to prove 
 the existence of TWS in case of~\eqref{RD:r:J*u-u},
 and~\eqref{RD:r:RieszFeller}$_{\RFtheta=0}$ with unbalanced potential~\cite{Gui+Zhao:2015}.

Chmaj~\cite{Chmaj:2013} also considers the TWP for~\eqref{RD:r:RieszFeller}$_{\RFtheta=0}$ with general bistable functions~$r$.
He approximates a given fractional Laplacian by a family of operators $J_\epsilon \ast u - (\int J_\epsilon)u$
 such that $\lim_{\epsilon\to 0} J_\epsilon \ast u - (\int J_\epsilon) u = \Riesz u$ in an appropriate sense.
This allows him to obtain a TWS of~\eqref{RD:r:RieszFeller}$_{\RFtheta=0}$ with general bistable function~$r$
 as the limit of the TWS~$u_\epsilon$ of~\eqref{RD:r:J*u-u} associated to $(J_\epsilon)_{\epsilon\geq 0}$. 
It might be possible to modify Chmaj's approach
 to study reaction-diffusion equation~\eqref{RD:r:RieszFeller} with asymmetric Riesz-Feller operators.
This would give an alternative existence proof for TWS in Theorem~\ref{thm:AK:RD:RieszFeller}.
However, Chen's approach allows to
 establish uniqueness (up to translations) and stability of TWS as well. 

\subsection{nonlocal Korteweg-de Vries-Burgers equation}
First we consider the integro-differential equation in multi-dimensions $d\geq 1$
\begin{equation} \label{KdVB:J*u-u}
 \sdiff{u}{t} + \sdiff{f(u)}{x} = \epsilon \Delta_x u + \gamma \epsilon^2 \sum_{j=1}^d \big( \phi_\epsilon \ast \sdiff{u}{x_j} -\sdiff{u}{x_j}\big)\ , \quad
  x\in\R^d \ , \quad t>0 \ , 
\end{equation}
 for parameters $\epsilon>0$, $\gamma\in\R$,
 a smooth even non-negative function $\phi$ with compact support and unit mass, i.e. $\integral{\R^d}{\phi(x)}{x}=1$,
 and the rescaled kernel function $\phi_\epsilon(x) = \phi(x/\epsilon) / \epsilon^d$.
It has been derived as a model for phase transitions with long range interactions close to the surface,
 which supports planar TWS associated to undercompressive shocks of~\eqref{SCL}, see~\cite{Rohde:2005}. 
A planar TWS $\TWS$ is a solution $u(x,t)=\profile(x-\speed t e)$ for some fixed vector $e\in\R^d$,
 such that the profile is transported in direction~$e$.
The existence of planar TWS is proven
 by reducing the problem to a one-dimensional TWP for~\eqref{KdVB:J*u-u}$_{d=1}$,
 identifying the associated reaction-diffusion equation~\eqref{RD:r:J*u-u}
 and using results in Proposition~\ref{prop:RD:J*u-u:TWS:existence}.
For cubic flux function $u^3$, the existence of planar TWS associated to undercompressive shocks of~\eqref{SCL} is established.
Moreover, the well-posedness of its Cauchy problem
 and the convergence of solutions $u^\epsilon$ as $\epsilon\searrow 0$ have been studied~\cite{Rohde:2005}.

\medskip \noindent
Another example is the \textbf{fractal Korteweg-de Vries-Burgers equation}
 \begin{equation}\label{fKdVB}
   \sdiff{u}{t} + \sdiff{f(u)}{x} = \epsilon \sdiff{}{x} \Dalpha u + \delta \sdifff{u}{x}{3}, \quad x\in\R, \quad t>0, 
 \end{equation}
 for some $\epsilon>0$ and $\delta\in\R$.

Equation \eqref{fKdVB} with $\alpha=1/3$ has been derived as a model for shallow water flows, 
 by performing formal asymptotic expansions associated to the triple-deck (boundary layer) theory
 in fluid mechanics, e.g. see \cite{Kluwick+etal:2010, Viertl:2005}.
In particular, the situations of one-layer and two-layer shallow water flows have been considered,
 which yield a quadratic (one layer) and cubic flux function (two layer), respectively.
In the monograph~\cite{Naumkin+Shishmarev:1994},
 similar models are considered and the well-posedness of the initial value problem and possible wave-breaking are studied.

The TWP for given endstates $\upm$
 is to study the existence of a TWS~$\TWS$ for~\eqref{fKdVB} in the sense of Definition~\ref{def:TWS}.
Such a TWS $(\profile,\RDspeed)$ satisfies the TWE
  \begin{equation} \label{fKdVB:TWE}
    h(\profile) := f(\profile) - f(\um) - c(\profile -\um)
     = \epsilon \Dalpha \profile + \delta \profile'' \ .
  \end{equation} 
We obtain a necessary condition for the existence of TWS -- see also~\eqref{eq:c+R} -- by multiplying the TWE with $\profile'$ and integrating on $\R$,
\begin{equation} \label{cond:ZeroSpeed}
   \integrall{\um}{\up}{ h(u) }{u} =
     \epsilon \integrall{-\infty}{\infty}{ \profile'\ \Dalpha \profile(\xi) }{\xi} \geq 0 \ ,
\end{equation}
where the last inequality follows from~\eqref{ineq:Caputo}.

\emph{connection with reaction-diffusion equation.}
If a TWS $\TWS$ for~\eqref{fKdVB} exists, 
 then $u(x,t)=\profile(x)$ is a stationary TWS $(\profile,0)$ of the evolution equation
 \begin{equation} \label{RD:multifractal} 
  \sdiff{u}{t} = -\epsilon \Dalpha u -\delta \sdifff{u}{x}{2} +h(u), \quad x\in\R, \quad t>0.
 \end{equation}
To interpret equation~\eqref{RD:multifractal} as a reaction-diffusion equation,
  we need to verify that $-\epsilon \Dalpha u -\delta \sdifff{u}{x}{2}$ is a diffusion operator,
  e.g. that $-\epsilon \Dalpha u -\delta \sdifff{u}{x}{2}$ generates a positivity preserving semigroup.

\begin{lemma}
Suppose $0<\alpha<1$ and $\gamma_1,\gamma_2 \in\R$.
The operator $\gamma_1 \Dalpha u + \gamma_2 \sdifff{u}{x}{2}$ is a L\'evy operator
 if and only if $\gamma_1\leq 0$ and $\gamma_2\geq 0$. 
Moreover, the associated heat kernel is strictly positive 
 if and only if $\gamma_2>0$. 
\end{lemma}
\begin{proof}
 For $\alpha\in (0,1)$, the operator $-\Dalpha$ is a Riesz-Feller operator~$\RieszFellerII{\alpha}{-\alpha}$
  and generates a positivity preserving convolution semigroup
  with a L\'evy stable probability distribution~$\GreenII{\alpha}{-\alpha}$ as its kernel.
  The probability distribution is absolutely continuous with respect to Lebesgue measure
   and its density has support on a half-line~\cite{Mainardi+Luchko+Pagnini:2001}.
  For example the kernel associated to $-\operator{D}^{1/2}$ is the L\'evy-Smirnov distribution.
 Thus, for $\gamma_1\leq 0$ and $\gamma_2\geq 0$, the operator $\gamma_1 \Dalpha u + \gamma_2 \sdifff{u}{x}{2}$ is a L\'evy operator,
  because it is a linear combination of L\'evy operators. 
 Using the notation for Fourier symbols of Riesz-Feller operators,
	the partial Fourier transform of equation
  \[ \sdiff{u}{t} = -\abs{\gamma_1} \mathcal{D}^{\alpha}[u] + \gamma_2 \sdifff{u}{x}{2} \]
  is given by
  $\sdiff{}{t} \Fourier[u](\FVar) = ( \abs{\gamma_1} \psi^\alpha_{-\alpha}(\FVar) - \gamma_2 \FVar^2 ) \Fourier[u](\FVar)$.
	Therefore, the operator generates a convolution semigroup with heat kernel
  \[
	   \Fourier^{-1}[\exp\{( \abs{\gamma_1} \psi^\alpha_{-\alpha}(\FVar) - \gamma_2 \FVar^2 )\ t\}] (x)
       = \GreenII{\alpha}{-\alpha} (\cdot,|\gamma_1| t) \ast \GreenII{2}{0} (\cdot,\gamma_2 t) \ (x) \ ,
  \]
  which is the convolution of two probability densities. 
  The kernel is positive on $\R$
   since probability densities are non-negative on $\R$ and the normal distribution~$\GreenII{2}{0}$ is positive on $\R$ for positive $\gamma_2 t$.

 The operator $\Dalpha$ for $\alpha\in(0,1)$ is not a Riesz-Feller operator, see Figure~\ref{fig:RieszFeller},
   and it generates a semigroup which is not positivity preserving.
 Thus it and any linear combination with $\gamma_1>0$ is not a L\'evy operator.

\qed
\end{proof}

\medskip \noindent
\textbf{convex flux functions.}
\begin{proposition}
Consider~\eqref{fKdVB} with $0<\alpha<1$, $\delta\in\R$ and strictly convex flux function~$f\in C^3(\R)$. 
For a shock triple $\ShockTriple$ satisfying the Rankine-Hugoniot condition~\eqref{RH},
 a non-constant TWS~$\TWS$ can exist if and only if Lax' entropy condition~\eqref{cond:entropy:Lax} is fulfilled, i.e. $\um>\up$. 
\end{proposition}
\begin{proof}
The Rankine-Hugoniot condition~\eqref{RH} ensures that 
 $h(u)$ in~\eqref{fKdVB:TWE} has exactly two roots~$\upm$.
If Lax' entropy condition~\eqref{cond:entropy:Lax} is fulfilled, 
 then $\um>\up$ and $-h(u)$ is monostable in the sense of Definition~\ref{def:reaction:type}.
Thus, the necessary condition~\eqref{cond:ZeroSpeed} is satisfied.
If $\um=\up$ then~\eqref{cond:ZeroSpeed} implies that $\profile$ is a constant function satisfying $\profile\equiv\upm$.
If $\um<\up$ then $-h(u)$ is monostable in the sense of Definition~\ref{def:reaction:type} with reversed roles of $\upm$.
Thus, the necessary condition~\eqref{cond:ZeroSpeed} is not satisfied.

\qed
\end{proof}

Next, we recall some existence result which have been obtained by directly studying the TWE.
In an Addendum~\cite{Cuesta+Achleitner:2017}, we removed an initial assumption on the solvability of the linearized TWE.
\begin{theorem}[\cite{Achleitner+Hittmeir+Schmeiser:2011}] \label{thm:AHS}
 Consider \eqref{fKdVB} with $\delta=0$ and convex flux function $f(u)$.
 For a shock triple $\ShockTriple$ satisfying~\eqref{RH} and \eqref{cond:entropy:Lax}, 
  there exists a monotone TWS of \eqref{fKdVB} in the sense of Definition~\ref{def:TWS}, 
  whose profile~$\profile\in C^1_b(\R)$ is unique (up to translations) among all functions
   $u\in \um + H^2(-\infty,0) \cap C^1_b(\R)$.
\end{theorem}
This positive existence result is consistent with the negative existence result in Proposition~\ref{prop:RD:fracLaplacian:TWS:existence} and Engler~\cite{Engler:2010} for \eqref{RD:r:RieszFeller} with non-extremal Riesz-Feller operators~$\RieszFeller$ for $(\RFalpha,\RFtheta)\in\RFnonextremal$.
The reason is that $-\Dalpha$ for $0<\alpha<1$ is the generator of a convolution semigroup with a one-sided strictly stable probability density function as its heat kernel; in contrast to heat kernels with genuine algebraic decay~\cite{Cabre+Roquejoffre:2009, Cabre+Roquejoffre:2013, Engler:2010}.

 \begin{theorem}[\cite{Achleitner+Cuesta+Hittmeir:2014}] \label{thm:ACH}
  Consider \eqref{fKdVB} with flux function $f(u) =u^2/2$.
	For a shock triple $\ShockTriple$ satisfying~\eqref{RH} and \eqref{cond:entropy:Lax}, 
 there exists a TWS of \eqref{fKdVB} in the sense of Definition~\ref{def:TWS}, 
  whose profile~$\profile$ is unique (up to translations) among all functions
   $u\in \um + H^4(-\infty,0) \cap C^3_b(\R)$.
 \end{theorem}
If dispersion dominates diffusion then the profile of a TWS $\TWS$ will be oscillatory in the limit $\xi\to\infty$.
For a classical KdVB equation this geometry of profiles depends on the ratio $\epsilon^2/\delta$
 and the threshold can be determined explicitly. 

\medskip \noindent
\textbf{concave-convex flux functions.}
We consider a cubic flux function $f(u)=u^3$ as the prototypical concave-convex flux function.
Again the necessary condition~\eqref{cond:ZeroSpeed} and the classification of function $h(u)=-r(u)$ in Figure~\ref{fig:classification:cubic}
 can be used to identify non-admissible shock triples $\ShockTriple$ for the TWP of~\eqref{fKdVB}.

We conjecture that a statement analogous to Proposition~\ref{prop:JMcKS} holds true.
Of special interest is again the occurrence of TWS $\TWS$ associated to non-classical shocks,
 which are only expected in case of~\eqref{fKdVB} with $\epsilon>0$ and $\delta>0$.

\begin{proposition}[conjecture] \label{prop:fKdVB:cubic:TWS:existence}
Suppose $f(u) =u^3$ and $\epsilon>0$.
\begin{enumerate}
\item If $\delta\leq 0$ then a TWS $(\profile,c)$ of \eqref{fKdVB} exists 
 if and only if $\ShockTriple$ satisfy the Rankine-Hugoniot condition~\eqref{RH} and the entropy condition~\eqref{cond:entropy:Lax}.
\item If $\delta> 0$ then a TWS $(\profile,c)$ of \eqref{fKdVB} exists for $\um>0$
 if and only if $\up\in S(\um)$ for some set $S(\um)$ similar to~\eqref{set:Sum:JMcKS}.
\end{enumerate}
\end{proposition}
\begin{proof}[sketch of proof]
If $\delta=0$, then equation~\eqref{fKdVB} is a viscous conservation law,
 and its TWE \eqref{fKdVB:TWE} is a fractional differential equation $\epsilon \Dalpha \profile = h(\profile)$.
Thus a heteroclinic orbit exists only for monostable $-h(u)$,
 i.e. if the unstable node $\um$ and the stable node $\up$ are not separated by any other root of $h$.
This follows from Theorem~\ref{thm:AHS} and its proof in~\cite{Achleitner+Hittmeir+Schmeiser:2011,Cuesta+Achleitner:2017}.

If $\delta<0$, then the TWE~\eqref{fKdVB:TWE} is associated to a reaction-diffusion equation~\eqref{RD:multifractal}
 via a stationary TWS ansatz $u(x,t) = \profile(x)$.
First we note that a stronger version of the necessary condition~\eqref{cond:ZeroSpeed} is available 
\begin{equation} \label{cond:ZeroSpeed:xi}
   \integrall{-\infty}{\xi}{ h(\profile) \profile' (y) }{y} =
     \epsilon \integrall{-\infty}{\xi}{ \profile'\ \Dalpha \profile(y) }{y} \geq 0 \ , \quad \forall \xi\in\R \ , 
\end{equation}
see~\cite{Achleitner+Cuesta+Hittmeir:2014}.
If $\up<\uStar<\um$ then $h(u)$ is an unstable reaction function, see Figure~\ref{fig:classification:cubic}. 
Thus there exists no TWS in the sense of Definition~\ref{def:TWS} satisfying the necessary condition~\eqref{cond:ZeroSpeed:xi}. 
If $\uStar<\up<\um$ then function $-h(u)$ is monostable in the sense of Definition~\ref{def:reaction:type}
 and the necessary condition~\eqref{cond:ZeroSpeed} can be satisfied. 
The existence of a TWS $\TWS$ can be proven by following the analysis in~\cite{Achleitner+Cuesta+Hittmeir:2014,Cuesta+Achleitner:2017}.
The TWP for other pairs $(\um,\up)$ is discussed similarly.

If $\delta>0$ then the occurrence of TWS $\TWS$ associated to non-classical shocks is possible.
Unlike in our previous examples, 
 the associated evolution equation~\eqref{RD:multifractal} is not a reaction-diffusion equation,
 since $-\epsilon\Dalpha\profile -\delta\profile''$ is not a L\'evy operator.
Especially, the results on existence of TWS for reaction-diffusion equations with bistable reaction function 
 can not be used to prove the existence of TWS $\TWS$ associated to a undercompressive shocks.
Instead, we investigate the TWP directly~\cite{Achleitner+Cuesta:201x},
 extending the analysis in~\cite{Achleitner+Cuesta+Hittmeir:2014,Cuesta+Achleitner:2017} for Burgers' flux to the cubic flux function $f(u)=u^3$.

\qed
\end{proof}


\subsection{Fowler's equation}
Fowler's equation~\eqref{dune} for dune formation is a special case of the evolution equation
\begin{equation} \label{Fowler}
  \sdiff{u}{t} + \sdiff{f(u)}{x} = \delta\sdifff{}{x}{2} u -\epsilon\sdiff{}{x} \Dalpha u \ , \quad t>0 \ , \quad x\in\R \ ,
\end{equation}
with $0<\alpha<1$, positive constant $\epsilon,\delta>0$ and flux function $f$.
Here the fractional derivative appears with the negative sign,
 but this instability is regularized by the second order derivative. 
The initial value problem for \eqref{dune} is well-posed in $L^2$~\cite{Alibaud+Azerad+Isebe:2010}.
However, it does not support a maximum principle,
 which is intuitive in the context of the application due to underlying erosions~\cite{Alibaud+Azerad+Isebe:2010}.  
The existence of TWS of~\eqref{dune} -- without assumptions~\eqref{TWS:limits} on the far-field behavior -- has been proven~\cite{Alvarez-Samaniego+Azerad:2009}.

For given endstates $\upm$, the TWP for~\eqref{Fowler} is
 to study the existence of a TWS~$\TWS$ for~\eqref{Fowler} in the sense of Definition~\ref{def:TWS}.
Such a TWS $(\profile,\RDspeed)$ satisfies the TWE
  \begin{equation} \label{Fowler:TWE}
    h(\profile) := f(\profile) - f(\um) - c(\profile -\um)
     = \delta\profile' - \epsilon\Dalpha \profile \ , \quad \xi\in\R \ .
  \end{equation} 
For $\delta=0$, the TWE reduces to a fractional differential equation $\epsilon\Dalpha \profile = -h(\profile)$,
 which has been analyzed in~\cite{Achleitner+Hittmeir+Schmeiser:2011,Cuesta+Achleitner:2017} for monostable functions~$-h(u)$.

Equation~\eqref{Fowler:TWE} is also the TWE for a TWS $(\profile,\delta)$ of an evolution equation 
\begin{equation} \label{RD:Fowler}
  \sdiff{u}{t} = -\epsilon\Dalpha u - h(u), \quad x\in\R, \quad t>0.
\end{equation}
For $\epsilon>0$, the operator is $-\epsilon \Dalpha \profile$ is a Riesz-Feller operator~$\epsilon \RieszFellerII{\alpha}{-\alpha}$
 whose heat kernel~$\GreenII{\alpha}{-\alpha}$ has only support on a halfline.
For a shock triple $\ShockTriple$ satisfying the Rankine-Hugoniot condition~\eqref{RH}, at least $h(\upm)=0$ holds.
Under these assumptions,
 equation~\eqref{RD:Fowler} is a reaction-diffusion equation
 with a Riesz-Feller operator modeling diffusion.

The abstract method in~\cite{Alvarez-Samaniego+Azerad:2009} does not provide any information on the far-field behavior.
Thus, assume the existence of a TWS $\TWS$ in the sense of Definition~\ref{def:TWS},
 for some shock triple $\ShockTriple$ satisfying the Rankine-Hugoniot condition~\eqref{RH}.
Again, a necessary condition is obtained by multiplying TWE~\eqref{Fowler:TWE} with $\profile'$ and integrating on $\R$;
hence, 
\begin{equation} \label{Fowler:TWS:NC}
 \integrall{\um}{\up}{h(u)}{u} = \integral{\R}{(\profile')^2}{\xi} - \integral{\R}{\profile' \Dalpha \profile}{\xi} \ .
\end{equation}
The left hand side is indefinite since each integral is non-negative, see also~\eqref{ineq:Caputo}.

For a cubic flux function $f(u)=u^3$ and a shock triple $\ShockTriple$ satisfying the Rankine-Hugoniot condition~\eqref{RH},
 we deduce a bistable reaction function $r(u)=-h(u)$ as long as $\up<-\up-\um<\um$ see Figure~\ref{fig:classification:cubic}.
However, since the heat kernel has only support on a halfline,
 we can not obtain a strict comparison principle as needed in Chen's approach~\cite{Chen:1997,Achleitner+Kuehn:2015,Achleitner+Kuehn:201x}.

\appendix

\section{Caputo fractional derivative on \texorpdfstring{$\R$}{R}}  
 \label{sec:Caputo}
 For $\alpha>0$, the (Gerasimov-)Caputo derivatives are defined as, see~\cite{Kilbas+etal:2006, Uchaikin:2013:bothVolumes},
 \begin{align*}
  (\Caputo{+} f)(x) &=
	  \begin{cases}
      f^{(n)}(x) & \text{if } \alpha=n\in\N_0 \ , \\
	    \tfrac{1}{\Gamma(n-\alpha)}\integrall{-\infty}{x}{ \frac{f^{(n)}(y)}{(x-y)^{\alpha-n+1}} }{y}  & \text{if } n-1 <\alpha<n \text{ for some } n\in\N_0 \ .
	  \end{cases}
	\\
	(\Caputo{-} f)(x) &= 
	  \begin{cases}
      f^{(n)}(x) & \text{if } \alpha=n\in\N_0 \ , \\
	    \tfrac{(-1)^n}{\Gamma(n-\alpha)}\integrall{x}{\infty}{ \frac{f^{(n)}(y)}{(y-x)^{\alpha-n+1}} }{y}  & \text{if } n-1 <\alpha<n \text{ for some } n\in\N_0 \ .
	  \end{cases}
 \end{align*}
 properties:
 \begin{itemize}
  \item For $\alpha>0$ and $\lambda> 0$
   \[ (\Caputo{+} \exp(\lambda \cdot))(x) = \lambda^\alpha  \exp(\lambda x) \ , \quad (\Caputo{-} \exp(-\lambda \cdot))(x) = \lambda^\alpha  \exp(-\lambda x) \]
  \item For $\alpha>0$ and $f\in\mathcal{S}(\R)$, a Caputo derivative is a Fourier multiplier operator with
   $(\Fourier \Caputo{+} f)(\FVar) = (\ii\FVar)^\alpha (\Fourier f)(\FVar)$
   where $(\ii\FVar)^\alpha = \exp(\alpha\pi \ii \sgn(\FVar)/2)$.
  \item If $\profile$ is the profile of a TWS~$\TWS$ in the sense of Definition~\ref{def:TWS}, then
\begin{equation}\label{ineq:Caputo}
 \integrall{-\infty}{\infty}{ \profile'(y)\, \Caputo{+} \profile(y) }{y} 
  = \tfrac12 \integral{\R}{ \profile'(x) \integral{\R}{ \frac{\profile'(y)}{|x-y|^\alpha} }{y} }{x} 
	\geq 0 \ , 
\end{equation}
where the last inequality follows from~\cite[Theorem 9.8]{Lieb+Loss:1997}.
\end{itemize}

\section{shock wave theory for scalar conservation laws} \label{sec:ShockWaveTheory}
A standard reference on the theory of conservation laws is~\cite{Dafermos:2010},
 whereas~\cite{LeFloch:2002} covers the special topic of non-classical shock solutions.
A scalar conservation law is a partial differential equation
 \begin{equation} \label{SCL}
  \sdiff{u}{t} + \sdiff{f(u)}{x} = 0 \ , \quad t>0 \ , \quad x\in\R \ , 
 \end{equation}
 for some flux function $f:\R\to\R$.
For nonlinear functions $f$, it is well known that 
 the initial value problem (IVP) for~\eqref{SCL} with smooth initial data
 may not have a classical solution for all time $t>0$
 (due to shock formation).
However, weak solutions may not be unique.
The \emph{Riemann problems} are a subclass of IVPs for~\eqref{SCL},
 and especially important in some numerical algorithms:
For given $\um, \up\in \R$, 
 find a weak solution $u(x,t)$ for the initial value problem of~\eqref{SCL} with initial condition 
 \begin{equation} \label{ID:RiemannProblem}
  u(x,0) = 
   \begin{cases}
    \um \ , & x<0 \ , \\
    \up \ , & x>0 \,. 
   \end{cases}
 \end{equation}
Weak solutions of a Riemann problem that are discontinuous for $t>0$ may not be unique.
\begin{example}
A \textit{shock wave} is a discontinuous solution of the Riemann problem,
\begin{equation} \label{SW}
u(x,t) = 
\begin{cases}
\um \ , & x<\speed t \ , \\
\up \ , & x>\speed t \ ,
\end{cases}
\end{equation}
if the \textit{shock triple} $\ShockTriple$ satisfies the Rankine-Hugoniot condition
\begin{equation} \label{RH} 
 f(\up)-f(\um) = \speed (\up -\um) \,.  
\end{equation}
The Rankine-Hugoniot condition~\eqref{RH} is a necessary condition
 that $\upm$ are stationary states of an associated TWE~\eqref{KdVB:TWE:local:integrated}, see~\eqref{cond:up:stationary}.
\end{example}

\subsection*{shock admissibility}
Classical approaches to select a unique weak solution of the Riemann problem are 
\begin{enumerate}[label=(\alph*)]
\item \textit{Lax' entropy condition:}
   \begin{equation} \label{cond:entropy:Lax} f'(\up)<\speed<f'(\um) \ . \end{equation} 
 It ensures that in the method of characteristics all characteristics enter the shock/discontinuity of a shock solution~\eqref{SW}.
 For convex flux function~$f$, condition~\eqref{cond:entropy:Lax} reduces to $\um>\up$.
 Shocks satisfying~\eqref{cond:entropy:Lax} are also called Lax or classical shocks.
 For non-convex flux functions~$f$, also non-classical shocks can arise in experiments,
  called slow undercompressive shocks if $f'(\upm)>\speed$, 
	and fast undercompressive shocks if $f'(\upm)<\speed$.

\item \textit{Oleinik's entropy condition.}
 \begin{equation}
  \label{Oleinik:EC} 
  \frac{f(w)-f(\um)}{w-\um} \geq \frac{f(\up)-f(\um)}{\up-\um} \XX{for all $w$ between $\um$ and $\up$.}
 \end{equation}
\item \textit{entropy solutions} satisfying integral inequalities based on entropy-entropy flux pairs, such as Kruzkov's family of entropy-entropy flux pairs.
\item \textit{vanishing viscosity.}
 In the classical vanishing viscosity approach,
  instead of~\eqref{SCL} one considers for $\epsilon>0$ equation
 \begin{equation} \label{SCL+VV} 
  \sdiff{u}{t} + \sdiff{f(u)}{x} = \epsilon \sdifff{u}{x}{2} \ , \quad t>0 \ , \quad x\in\R \ , 
 \end{equation}
 where $\epsilon \sdifff{u}{x}{2}$ models diffusive effects such as friction.
 Equation~\eqref{SCL+VV} is a parabolic equation, 
  hence the Cauchy problem has global smooth solutions $u^\epsilon$ for positive times,
	especially for Riemann data~\eqref{ID:RiemannProblem}.
 An admissible weak solution of the Riemann problem is identified by studying the limit of $u^\epsilon$ as $\epsilon\searrow 0$.
 \newline
 In other applications, different higher order effects may be important. 
 For example, a nonlocal generalized KdVB equation~\eqref{KdVB:f:nonlocal} can be interpreted
  as a scalar conservation law~\eqref{SCL} with higher-order effects $\operator{R}[u] := \epsilon \Levy_1[u] +\delta \sdiff{}{x} \Levy_2[u]$.
 \newline 

Already for convex functions~$f$,
 the convergence of solutions of the regularized equations (e.g. \eqref{KdVB:f:nonlocal}) to solutions of~\eqref{SCL}
 reveals a diverse solution structure.
The solutions of viscous conservation laws~\eqref{SCL+VV} converge for $\epsilon\searrow 0$ to Kruzkov entropy solutions of~\eqref{SCL}.
In contrast, in case of KdVB equation~\eqref{KdVB:f:local} the limit $\epsilon,\delta\to 0$ depends on the relative strength of diffusion and dispersion:
 \begin{itemize}
  \item \textbf{weak dispersion} $\delta=O(\epsilon^2)$ for $\epsilon\to 0$
	 e.g. $\delta=\beta\epsilon^2$ for some $\beta>0$. \\
	 TWS converge strongly to entropy solution of Burgers equation.
  \item \textbf{moderate dispersion} $\delta=o(\epsilon)$ for $\epsilon\to 0$
	 includes weak dispersion. \\
	 TWS converge strongly to entropy solution of Burgers equation, see~\cite{Perthame+Ryzhik:2007}.
	\item \textbf{strong dispersion}
	 weak limit of TWS for $\epsilon,\delta\to 0$
	  may not be a weak solution of Burgers equation. 
 \end{itemize}

\medskip 
For non-convex flux functions~$f$,
 a TWS may converge to a weak solution of~\eqref{SCL}
 which is not an Kruzkov entropy solution,  
 but a non-classical shock.
\end{enumerate}

A simplistic shock admissibility criterion based on the vanishing viscosity approach is the existence of TWS for a given shock triple: 
\begin{definition}[compare with \cite{Jacobs+McKinney+Shearer:1995}]
A solution $u$ of the Riemann problem is called \textit{admissible} (with respect to a fixed regularization~$\operator{R}$),
if there exists a TWS $\TWS$ in the sense of Definition~\ref{def:TWS} of the regularized equation (e.g. \eqref{KdVB:f:nonlocal}) 
 for every shock wave with shock triple $\ShockTriple$ in the solution $u$.
\end{definition}

\bibliographystyle{abbrv}

\begin{thebibliography}{10}

\bibitem{Achleitner+Cuesta:201x}
F.~Achleitner and C.~M. Cuesta.
\newblock Non-classical shocks in a non-local generalised {K}orteweg-de
  {V}ries-{B}urgers equation.
\newblock work in progress.

\bibitem{Achleitner+Cuesta+Hittmeir:2014}
F.~Achleitner, C.~M. Cuesta, and S.~Hittmeir.
\newblock Travelling waves for a non-local {K}orteweg--de {V}ries--{B}urgers
  equation.
\newblock {\em J. Differential Equations}, 257(3):720--758, 2014.

\bibitem{Achleitner+Hittmeir+Schmeiser:2011}
F.~Achleitner, S.~Hittmeir, and C.~Schmeiser.
\newblock On nonlinear conservation laws with a nonlocal diffusion term.
\newblock {\em J. Differential Equations}, 250(4):2177--2196, 2011.

\bibitem{Achleitner+Kuehn:201x}
F.~Achleitner and C.~Kuehn.
\newblock Traveling waves for a bistable reaction-diffusion equation with
  nonlocal {R}iesz-{F}eller operator.
\newblock work in progress.

\bibitem{Achleitner+Kuehn:2015:BCAM}
F.~Achleitner and C.~Kuehn.
\newblock Analysis and numerics of traveling waves for asymmetric fractional
  reaction-diffusion equations.
\newblock {\em Commun. Appl. Ind. Math.}, 6(2):e--532, 25, 2015.

\bibitem{Achleitner+Kuehn:2015}
F.~Achleitner and C.~Kuehn.
\newblock Traveling waves for a bistable equation with nonlocal diffusion.
\newblock {\em Adv. Differential Equations}, 20(9-10):887--936, 2015.

\bibitem{Alibaud+Azerad+Isebe:2010}
N.~Alibaud, P.~Azerad, D.~Is{\`e}be, et~al.
\newblock A non-monotone nonlocal conservation law for dune morphodynamics.
\newblock {\em Differential and Integral Equations}, 23(1/2):155--188, 2010.

\bibitem{Alvarez-Samaniego+Azerad:2009}
B.~Alvarez-Samaniego and P.~Azerad.
\newblock Existence of travelling-wave solutions and local well-posedness of
  the {F}owler equation.
\newblock {\em Discrete Contin. Dyn. Syst. Ser. B}, 12(4):671--692, 2009.

\bibitem{Applebaum:2009}
D.~Applebaum.
\newblock {\em L\'evy processes and stochastic calculus}, volume 116 of {\em
  Cambridge Studies in Advanced Mathematics}.
\newblock Cambridge University Press, Cambridge, second edition, 2009.

\bibitem{Aronson+Weinberger:1974}
D.~G. Aronson and H.~F. Weinberger.
\newblock Nonlinear diffusion in population genetics, combustion, and nerve
  pulse propagation.
\newblock In {\em Partial differential equations and related topics}, pages
  5--49. Lecture Notes in Math., Vol. 446. Springer, Berlin, 1975.

\bibitem{Bates+etal:1997}
P.~W. Bates, P.~C. Fife, X.~Ren, and X.~Wang.
\newblock {Traveling waves in a convolution model for phase transitions}.
\newblock {\em Archive for Rational Mechanics and Analysis}, 138:105--136,
  1997.

\bibitem{Bona+Schonbek:1985}
J.~L. Bona and M.~E. Schonbek.
\newblock Travelling-wave solutions to the {K}orteweg-de {V}ries-{B}urgers
  equation.
\newblock {\em Proc. Roy. Soc. Edinburgh Sect. A}, 101(3-4):207--226, 1985.

\bibitem{Cabre+Roquejoffre:2009}
X.~Cabr{\'e} and J.-M. Roquejoffre.
\newblock Propagation de fronts dans les {\'e}quations de fisher--kpp avec
  diffusion fractionnaire.
\newblock {\em Comptes Rendus Math{\'e}matique}, 347(23):1361--1366, 2009.

\bibitem{Cabre+Roquejoffre:2013}
X.~Cabr{\'e} and J.-M. Roquejoffre.
\newblock The influence of fractional diffusion in fisher-kpp equations.
\newblock {\em Communications in mathematical physics}, 320(3):679--722, 2013.

\bibitem{Cabre+Sire:2014}
X.~Cabr{\'e} and Y.~Sire.
\newblock Nonlinear equations for fractional {L}aplacians, {I}: {R}egularity,
  maximum principles, and {H}amiltonian estimates.
\newblock {\em Ann.Inst.H.Poincar\'e Anal.Non Lin}, 31(1):23--53, 2014.

\bibitem{Cabre+Sire:2015}
X.~Cabr{\'e} and Y.~Sire.
\newblock Nonlinear equations for fractional {L}aplacians {II}: {E}xistence,
  uniqueness, and qualitative properties of solutions.
\newblock {\em Trans. Amer. Math. Soc.}, 367(2):911--941, 2015.

\bibitem{Cabre+SolaMorales:2005}
X.~Cabr\'{e} and J.~Sol\`{a}-Morales.
\newblock {Layer solutions in a half-space for boundary reactions}.
\newblock {\em Communications on Pure and Applied Mathematics}, 58:1678--1732,
  2005.

\bibitem{Caffarelli+Silvestre:2007}
L.~Caffarelli and L.~Silvestre.
\newblock An extension problem related to the fractional {L}aplacian.
\newblock {\em Comm. Partial Differential Equations}, 32(7-9):1245--1260, 2007.

\bibitem{Chen:1997}
X.~Chen.
\newblock Existence, uniqueness, and asymptotic stability of travelling waves
  in nonlocal evolution equations.
\newblock {\em Adv. Differential Equations}, 2:125--160, 1997.

\bibitem{Chmaj:2013}
A.~Chmaj.
\newblock {Existence of traveling waves in the fractional bistable equation}.
\newblock {\em Archiv der Mathematik}, 100(5):473--480, May 2013.

\bibitem{Coville+Dupaigne:2007}
J.~Coville and L.~Dupaigne.
\newblock On a non-local equation arising in population dynamics.
\newblock {\em Proc. Roy. Soc. Edinburgh Sect. A}, 137(4):727--755, 2007.

\bibitem{Cuesta+Achleitner:2017}
C.~M. Cuesta and F.~Achleitner.
\newblock Addendum to ``{T}ravelling waves for a non-local {K}orteweg--de
  {V}ries--{B}urgers equation'' [{J}. {D}ifferential {E}quations 257 (3) (2014)
  720--758].
\newblock {\em J. Differential Equations}, 262(2):1155--1160, 2017.

\bibitem{Dafermos:2010}
C.~M. Dafermos.
\newblock {\em Hyperbolic conservation laws in continuum physics}, volume 325
  of {\em Grundlehren der Mathematischen Wissenschaften}.
\newblock Springer-Verlag, Berlin, 2010.

\bibitem{Engler:2010}
H.~Engler.
\newblock On the speed of spread for fractional reaction-diffusion equations.
\newblock {\em Int. J. Differ. Equ.}, pages Art. ID 315421, 16, 2010.

\bibitem{Fife+McLeod:1977}
P.~Fife and J.~McLeod.
\newblock The approach of solutions nonlinear diffusion equations to travelling
  front solutions.
\newblock {\em Arch. Rational Mech. Anal.}, 65:335--361, 1977.

\bibitem{Fowler:2002}
A.~C. Fowler.
\newblock Evolution equations for dunes and drumlins.
\newblock {\em RACSAM. Rev. R. Acad. Cienc. Exactas F\'\i s. Nat. Ser. A Mat.},
  96(3):377--387, 2002.

\bibitem{Gui+Zhao:2015}
C.~Gui and M.~Zhao.
\newblock Traveling wave solutions of {A}llen-{C}ahn equation with a fractional
  {L}aplacian.
\newblock {\em Ann. Inst. H. Poincar\'e Anal. Non Lin\'eaire}, 32(4):785--812,
  2015.

\bibitem{Hayes+LeFloch:1997}
B.~T. Hayes and P.~G. LeFloch.
\newblock {Non-classical shocks and kinetic relations: scalar conservation
  laws}.
\newblock {\em Arch. Rational Mech. Anal.}, 139(1):1--56, 1997.

\bibitem{Hayes+Shearer:1999}
B.~T. Hayes and M.~Shearer.
\newblock Undercompressive shocks and {R}iemann problems for scalar
  conservation laws with non-convex fluxes.
\newblock {\em Proc.Roy.Soc.Edinburgh Sect.A}, 129:733--754, 1999.

\bibitem{Jacob:2001}
N.~Jacob.
\newblock {\em Pseudo differential operators and {M}arkov processes. {V}ol.
  {I}}.
\newblock Imperial College Press, London, 2001.
\newblock Fourier analysis and semigroups.

\bibitem{Jacobs+McKinney+Shearer:1995}
D.~Jacobs, B.~McKinney, and M.~Shearer.
\newblock {Travelling wave solutions of the modified
  {K}orteweg-de{V}ries-{B}urgers Equation}.
\newblock {\em Journal of Differential Equations}, 166:448--467, 1995.

\bibitem{Kilbas+etal:2006}
A.~A. Kilbas, H.~M. Srivastava, and J.~J. Trujillo.
\newblock {\em Theory and applications of fractional differential equations}.
\newblock Elsevier Science B.V., Amsterdam, 2006.

\bibitem{Kluwick+etal:2010}
A.~Kluwick, E.~A. Cox, A.~Exner, and C.~Grinschgl.
\newblock On the internal structure of weakly nonlinear bores in laminar high
  reynolds number flow.
\newblock {\em Acta Mech}, 210(1):135--157, 2010.

\bibitem{LeFloch:2002}
P.~G. LeFloch.
\newblock {\em Hyperbolic systems of conservation laws}.
\newblock Lectures in Mathematics ETH Z\"urich. Birkh\"auser Verlag, Basel,
  2002.
\newblock The theory of classical and nonclassical shock waves.

\bibitem{Lieb+Loss:1997}
E.~H. Lieb and M.~Loss.
\newblock {\em Analysis}, volume~14 of {\em Graduate Studies in Mathematics}.
\newblock American Mathematical Society, Providence, RI, 1997.

\bibitem{Mainardi+Luchko+Pagnini:2001}
F.~Mainardi, Y.~Luchko, and G.~Pagnini.
\newblock The fundamental solution of the space-time fractional diffusion
  equation.
\newblock {\em Fract. Calc. Appl. Anal.}, 4(2):153--192, 2001.

\bibitem{Mendez+etal:2010}
V.~M{\'e}ndez, S.~Fedotov, and W.~Horsthemke.
\newblock {\em Reaction-transport systems}.
\newblock Springer Series in Synergetics. Springer, Heidelberg, 2010.

\bibitem{Naumkin+Shishmarev:1994}
P.~I. Naumkin and I.~A. Shishmar{\"e}v.
\newblock {\em Nonlinear nonlocal equations in the theory of waves}.
\newblock American Mathematical Society, Providence, RI, 1994.

\bibitem{Palatucci+etal:2013}
G.~Palatucci, O.~Savin, and E.~Valdinoci.
\newblock Local and global minimizers for a variational energy involving a
  fractional norm.
\newblock {\em Ann. Mat. Pura Appl. (4)}, 192(4):673--718, 2013.

\bibitem{Perthame+Ryzhik:2007}
B.~Perthame and L.~Ryzhik.
\newblock Moderate dispersion in scalar conservation laws.
\newblock {\em Communications in Mathematical Sciences}, 5:473--484, 2007.

\bibitem{Rohde:2005}
C.~Rohde.
\newblock Scalar conservation laws with mixed local and nonlocal
  diffusion-dispersion terms.
\newblock {\em SIAM Journal on Mathematical Analysis}, 37(1):103--129, 2005.

\bibitem{Sato:1999}
K.-i. Sato.
\newblock {\em L\'evy processes and infinitely divisible distributions},
  volume~68 of {\em Cambridge Studies in Advanced Mathematics}.
\newblock Cambridge University Press, Cambridge, 1999.

\bibitem{Uchaikin:2013:bothVolumes}
V.~V. Uchaikin.
\newblock {\em Fractional derivatives for physicists and engineers. {V}olumes
  {I \& II}}.
\newblock Nonlinear Physical Science. Higher Education Press, Beijing;
  Springer, Heidelberg, 2013.

\bibitem{Viertl:2005}
N.~Viertl.
\newblock {\em Viscous regularisation of weak laminar hydraulic jumps and bore
  in two layer shallow water flow}.
\newblock Phd thesis, Technische Universit\"at Wien, 2005.

\bibitem{Volpert3:1994}
A.~I. Volpert, V.~A. Volpert, and V.~A. Volpert.
\newblock {\em Traveling wave solutions of parabolic systems}.
\newblock American Mathematical Society, Providence, RI, 1994.

\bibitem{Volpert+etal:2013}
V.~A. Volpert, Y.~Nec, and A.~A. Nepomnyashchy.
\newblock Fronts in anomalous diffusion-reaction systems.
\newblock {\em Philos.Trans.R.Soc.Lond.Ser.A Math.Phys.Eng.Sci.},
  371(1982):20120179, 18, 2013.

\bibitem{Yagisita:2009:monostable}
H.~Yagisita.
\newblock Existence and nonexistence of traveling waves for a nonlocal
  monostable equation.
\newblock {\em Publ. Res. Inst. Math. Sci.}, 45(4):925--953, 2009.

\end{thebibliography}

\end{document}